\documentclass[11pt,a4paper,reqno]{amsart}
\pdfoutput=1
\usepackage{amsmath,amsthm,latexsym,amssymb,wasysym,mathtools}
\usepackage[margin=1.0in,tmargin=0.9in,bmargin=0.80in]{geometry}
\usepackage{hyperref}
\usepackage{bbm}
\usepackage{xargs}
\usepackage{enumerate}
\usepackage{enumitem}
\setlist[enumerate]{label={\upshape(\roman*)}}
\usepackage[textsize=footnotesize]{todonotes}
\usepackage[flushleft]{threeparttable}
\usepackage{multirow}
\usepackage{longtable}
\usepackage{times}
\usepackage{microtype}
\usepackage[numbers]{natbib}
\usepackage{xcolor}
\usepackage{tabularx}
\usepackage{ltablex}

\definecolor{bred}{rgb}{0.8,0,0}
\hypersetup{colorlinks,linkcolor={blue},citecolor={bred},urlcolor={blue}}

\newtheorem{theorem}{Theorem}[section]
\newtheorem{proposition}[theorem]{Proposition}
\newtheorem{lemma}[theorem]{Lemma}
\newtheorem{corollary}[theorem]{Corollary}
\newtheorem{remark}[theorem]{Remark}

\newtheorem{assumption}{Assumption}

\def\cT{\mathsf{T}}
\def\cK{\mathsf{K}}
\def\cI{\mathsf{I}}
\def\cL{\mathsf{L}}
\def\cM{\mathsf{M}}
\def\cS{\mathsf{S}}
\def\cC{\mathsf{C}}
\def\cc{\mathsf{c}}

\def\rmd{\mathrm{d}}
\def\rmKL{\mathrm{KL}}
\def\rmLS{\mathrm{LS}}
\def\rmdet{\mathrm{det}}

\newcommand{\E}{\mathbb{E}}
\newcommand{\N}{\mathbb{N}}
\newcommand{\R}{\mathbb{R}}
\newcommand{\1}{\mathbbm{1}}
\newcommand{\lfrf}[1]{\left\lfloor #1 \right\rfloor}
\newcommand{\lcrc}[1]{\left\lceil #1 \right\rceil}

\allowdisplaybreaks

\begin{document}

\title[]{kTULA: A Langevin sampling algorithm with improved KL bounds \\under super-linear log-gradients}

\author[I. Lytras]{Iosif Lytras}
\author[S. Sabanis]{Sotirios Sabanis}
\author[Y. Zhang]{Ying Zhang}

\address {Archimedes/Athena Research Centre  \& National Technical University of Athens, Athens, Greece,Athens, Greece}
\email{i.lytras@athenarc.gr}

\address{School of Mathematics, The University of Edinburgh, Edinburgh, UK \& The Alan Turing Institute, London, UK \& National Technical University of Athens, Athens, Greece}
\email{s.sabanis@ed.ac.uk}

\address{Financial Technology Thrust, Society Hub, The Hong Kong University of Science and Technology (Guangzhou), 
Guangzhou, China}
\email{yingzhang@hkust-gz.edu.cn}

\date{}
\thanks{
Financial support by the Guangzhou-HKUST(GZ) Joint Funding Programs (No. 2024A03J0630 and No. 2025A03J3322), project MIS 5154714 of the National Recovery and Resilience Plan Greece 2.0 funded by the European Union under the NextGenerationEU Program, and a grant from the Simons Foundation is gratefully acknowledged.}

\begin{abstract}
Motivated by applications in deep learning, where the global Lipschitz continuity condition is often not satisfied, we examine the problem of sampling from distributions with super-linearly growing log-gradients. We propose a novel tamed Langevin dynamics-based algorithm, called kTULA, to solve the aforementioned sampling problem, and provide a theoretical guarantee for its performance. More precisely, we establish a non-asymptotic convergence bound in Kullback-Leibler (KL) divergence with the best-known rate of convergence equal to $2-\overline{\epsilon}$, $\overline{\epsilon}>0$, which significantly improves relevant results in existing literature. This enables us to obtain an improved non-asymptotic error bound in Wasserstein-2 distance, which can be used to further derive a non-asymptotic guarantee for kTULA to solve the associated optimization problems. To illustrate the applicability of kTULA, we apply the proposed algorithm to the problem of sampling from a high-dimensional double-well potential distribution and to an optimization problem involving a neural network. We show that our main results can be used to provide theoretical guarantees for the performance of kTULA. 
\end{abstract}
\maketitle

\section{Introduction}
A wide range of contemporary and emerging machine learning methods depend critically on the ability to efficiently sample from a high-dimensional target distribution $\pi_{\beta}$ on $\R^d$, typically expressed in Gibbs form as $\pi_{\beta}(\theta) \propto \exp(-\beta u(\theta))$, where $u\colon \R^d \to \R$ is a given potential function and $\beta>0$ is the inverse temperature parameter. Under suitable regularity conditions on $u$, the distribution $\pi_{\beta}$ naturally emerges as the stationary distribution of the Langevin stochastic differential equation (SDE) given by
\begin{equation}
\label{eq:LSDE}
Z_0 \coloneqq  \theta_0, \quad \rmd Z_t=-h(Z_t) \rmd t+ \sqrt{2\beta^{-1}} \rmd B_t, \quad t\geq 0,
\end{equation}
where $\theta_0$ is an $\R^d$-valued random variable, $h = \nabla u$, and $(B_t)_{t\geq 0}$ denotes a standard Brownian motion in $\R^d$. Leveraging this connection, one of the most widely used algorithms for sampling from $\pi_{\beta}$ is the unadjusted Langevin algorithm (ULA), denoted by $(\theta_n^{\text{ULA}})_{n\in\N_0}$, which is defined via the recursion:
\begin{equation}
\label{eq:ULA}
\theta_0^{\text{ULA}}\coloneqq  \theta_0, \quad \theta_{n+1}^{\text{ULA}} = \theta_n^{\text{ULA}} - \lambda h(\theta_n^{\text{ULA}}) + \sqrt{2\lambda\beta^{-1}} \xi_{n+1},
\end{equation}
where $\lambda > 0$ is the stepsize and $(\xi_n)_{n\in\N_0}$ is an i.i.d.\ sequence of standard Gaussian vectors in $\R^d$. This iteration can be interpreted as an Euler–Maruyama discretization of SDE~\eqref{eq:LSDE}, thus, for sufficiently small $\lambda$ and large $n$, the distribution of $\theta_n^{\text{ULA}}$ approximates the target distribution $\pi_{\beta}$.

This basic idea has spurred extensive research into the non-asymptotic convergence properties of ULA~\eqref{eq:ULA}, particularly in probability metrics such as the Wasserstein distance, total variation (TV) distance, and divergences like Kullback-Leibler (KL) or Rényi divergence. Much of this work assumes that $\pi_{\beta}$ is log-concave with Lipschitz continuous log-density gradient, i.e., that $u$ is convex and Lipschitz smooth, see, for instance, \cite{convex,dalalyan2017theoretical,durmus2017nonasymptotic,durmus2019high}. 

Beyond these settings, especially in the presence of multi-modal distributions, significant effort has been put into relaxing the convexity condition. This includes assuming weaker structural properties such as convexity at infinity and dissipativity, meaning that $u$ is convex outside a compact region, and its gradient $h$ satisfies a coercivity condition of the form $\langle h(\theta), \theta \rangle = \Omega(|\theta|^2)$, $\theta\in\R^d$. Key contributions in this direction include \cite{berkeley,erdogdu2022convergence}, along with related advances for the Stochastic Gradient Langevin Dynamics (SGLD) algorithm in \cite{nonconvex,raginsky,sabanis2020fully,xu2018global,zhang2023nonasymptotic}.

In parallel, building on the insights of \cite{vempala2019rapid}, another line of work investigates how isoperimetric inequalities can be used to analyze the convergence of ULA~\eqref{eq:ULA}, particularly under the assumption that $u$ is Lipschitz smooth. Notably, some results leverage logarithmic Sobolev inequalities (LSI) as in \cite{chewi2021analysis,mou2022improved}, while others use Poincaré inequalities (PI) as in \cite{balasubramanian2022towards}, or even weaker functional inequalities, as explored in \cite{mousavi2023towards}. 

This work aims to unify two major directions in the sampling literature: one that relaxes the global Lipschitz smoothness assumption, and another that weakens convexity requirements using isoperimetric inequalities. Our goal is to combine the strengths of both approaches. Inspired by deep learning applications—where objective functions are often highly non-Lipschitz—we pose the following central question:
\begin{center}\itshape
\vspace{0.5em}
How can we perform efficient sampling when the target distribution is neither log-concave nor associated with linearly growing gradients?
\vspace{0.5em}
\end{center}
This scenario presents substantial challenges. As noted in several prior studies, both ULA \eqref{eq:ULA} and its stochastic variants (e.g., SGLD) can become unstable in such regimes. In particular, when $h$ grows super-linearly, the associated Euler–Maruyama scheme—which forms the backbone of ULA~\eqref{eq:ULA}—can fail dramatically. A pivotal result in \cite{hutzenthaler2011} showed that, in such cases, the numerical approximation can diverge from the true SDE solution in mean square sense, even over finite time horizons. This divergence is linked directly to the explosion of moments in the discretized process, highlighting why standard schemes can break down when applied to approximate SDE \eqref{eq:LSDE} with super-linearly growing $h$.

Addressing this issue requires a fundamentally different approach—one that revisits the way that ULA~\eqref{eq:ULA} is designed as a numerical discretization of SDE~\eqref{eq:LSDE}, and leverages insights from the theory of numerical methods for SDEs. In this vein, a promising class of techniques—known as tamed Euler schemes—emerged in \cite{hutzenthaler2012} and subsequent studies such as \cite{tamed-euler,SabanisAoAP}. These methods modify $h$ to ensure the stability of ULA~\eqref{eq:ULA} even under super-linear growth. Specifically, they replace the original coefficient $h$ with a modified version $h_{\lambda}$, which depends on the stepsize $\lambda$ and is constructed to satisfy two essential properties:
\begin{enumerate}
[label=\textbf{\upshape(P\arabic*)}]
\item \label{enum:prop1}
The tamed coefficient $h_{\lambda}$ grows at most linearly, i.e., $h_{\lambda}(\theta) = \mathcal{O}(|\theta|)$ as $|\theta| \to \infty$.
\item \label{enum:prop2}
$h_{\lambda}$ converges pointwise to the original coefficient $h$ as $\lambda \to 0$.
\end{enumerate}
These features ensure both the stability of the numerical scheme and the consistency with the original dynamics as the stepsize vanishes.

The taming technique has been applied previously in the setting of Langevin dynamics-based sampling in multiple works under strong dissipativity or convexity assumptions \cite{tula,johnston2023kinetic,hola}, in the stochastic gradient case \cite{lim2021nonasymptotic,lim2024langevin,TUSLA}, under a ``convexity at infinity'' assumption \cite{neufeld2024non,neufeld2025non}, and, in recent works, under a functional inequality in \cite{lytras2024tamed,lytras2025taming}, while some important work has been done using truncated or projected schemes under a convexity at infinity condition \cite{bao2024}.

Among all this important work, \cite{lytras2024tamed,lytras2025taming} entail the only results in KL divergence (which is a considerably stronger metric than Wasserstein and total variation) and achieve a $\mathcal{O}(\epsilon^{-1})$ rate. In this work, we propose a new Langevin dynamics-based algorithm, called the KL-accelerated tamed unadjusted Langevin algorithm (kTULA), which can be viewed as a modification of TULA originally proposed in \cite{tula}. Our goal is to establish a convergence result of kTULA in KL divergence with an improved rate of convergence under relaxed conditions. To this end, we first note that kTULA satisfies the aforementioned properties \ref{enum:prop1} and \ref{enum:prop2}. Then, under the conditions that the Hessian of $u$ is polynomially Lipschitz continuous, $h$ satisfies a dissipativity condition, $\pi_{\beta}$ satisfies a Log-Sobolev inequality, together with certain conditions on the distribution of $\theta_0$, we obtain a non-asymptotic error bound in KL divergence of the distribution of kTULA w.r.t.\ $\pi_{\beta}$. Crucially, the rate of convergence is proved to be $2-\overline{\epsilon}$ with $\overline{\epsilon}>0$, which can be made arbitrarily close to $2$. This provides the state-of-the-art convergence result for the Langevin dynamics-based algorithms to sample from distributions with highly non-linear potentials. Consequently, we are able to derive a non-asymptotic convergence bound in Wasserstein-2 distance via Talagrand’s inequality and thus a non-asymptotic upper estimate for the expected excess risk of the associated optimization problem (linked to the problem of sampling from $\pi_{\beta}$). Furthermore, to illustrate the applicability of our newly proposed algorithm, we consider using kTULA to sample from a high-dimensional double-well potential distribution and to solve an optimization problem involving a neural network. We show that our theoretical results can be used to provide convergence guarantees for kTULA to solve the aforementioned problems.

We conclude this section by introducing some notation. Let $(\Omega,\mathcal{F},P)$ be a probability space. We denote by $\E[Z]$  the expectation of a random variable $Z$. For $1\leq p<\infty$, $L^p$ is used to denote the usual space of $p$-integrable real-valued random variables. Fix integers $d, m \geq 1$. For an $\R^d$-valued random variable $Z$, its law on $\mathcal{B}(\R^d)$, i.e. the Borel sigma-algebra of $\R^d$, is denoted by $\mathcal{L}(Z)$. For a positive real number $a$, we denote by $\lfrf{a}$ its integer part, and $\lcrc{a} \coloneqq \min\{b\in \mathbb{Z}|b\geq a\} $. The Euclidean scalar product is denoted by $\langle \cdot,\cdot\rangle$, with $|\cdot|$ standing for the corresponding norm (where the dimension of the space may vary depending on the context). Denote by $\|M\|$ and $M^{\cT}$ the spectral norm and the transpose of any given matrix $M \in \R^{d\times m}$, respectively. Let $f\colon\R^d \rightarrow \R$ and $g\colon \R^d \rightarrow \R^d$ be twice continuously differentiable functions. Denote by $\nabla f$ and $\nabla^2 f$ the gradient of $f$ and the Hessian of $f$, respectively. For any integer $q \geq 1$, let $\mathcal{P}(\R^q)$ denote the set of probability measures on $\mathcal{B}(\R^q)$. For $\mu\in\mathcal{P}(\R^d)$ and for a $\mu$-integrable function $f\colon\R^d\to\R$, the notation $\mu(f)\coloneqq\int_{\R^d} f(\theta)\mu(\rmd \theta)$ is used. Let $\mu,\nu\in\mathcal{P}(\R^d)$ be two probability measures. If $\mu \ll \nu$, we denote by $\rmd \mu/\rmd \nu$ the Radon-Nikodym derivative of $\mu $ w.r.t.\ $\nu$, and the KL divergence of $\mu $ w.r.t.\ $\nu$ is given by 
\[
\rmKL(\mu\|\nu) = \int_{\R^d} \frac{\rmd \mu}{\rmd \nu} \log\left(\frac{\rmd \mu}{\rmd \nu} \right)\, \rmd\nu.
\]
Let $\mathcal{C}(\mu,\nu)$ denote the set of probability measures $\zeta$ on $\mathcal{B}(\R^{2d})$ such that its respective marginals are $\mu,\nu$. For two Borel probability measures $\mu$ and $\nu$ defined on $\R^d$ with finite $p$-th moments, the Wasserstein distance of order $p \geq 1$ is defined as
\[
{W}_p(\mu,\nu)\coloneqq
\left(\inf_{\zeta\in\mathcal{C}(\mu,\nu)}\int_{\R^d}\int_{\R^d}|\theta-\overline{\theta}|^p\zeta(\rmd \theta \rmd \overline{\theta})\right)^{1/p}.
\]

\section{Assumptions and main results}\label{sec:main}

Let $u \colon \R^d \to \R$ be a three times continuously differentiable function satisfying $\int_{\R^d} e^{-\beta u(\theta)} \, \rmd \theta <\infty$ for any $\beta>0$. Denote by $h \coloneqq \nabla u$ and $H\coloneqq\nabla^2 u$ the gradient and Hessian of $u$, respectively. Furthermore, define, for any $\beta>0$,
\begin{equation}\label{eq:pibetaexp}
\pi_{\beta}(A) \coloneqq \frac{\int_A e^{-\beta u(\theta)} \, \rmd \theta}{\int_{\R^d} e^{-\beta u(\theta)} \, \rmd \theta}, \quad A \in \mathcal{B}(\R^d).
\end{equation}

To sample from $\pi_{\beta}$, we fix $\beta>0$ and propose the kTULA algorithm given by
\begin{equation}\label{eq:ktula}
\theta^{\lambda}_0 \coloneqq \theta_0, \quad \theta^{\lambda}_{n+1}=\theta^{\lambda}_n-\lambda h_\lambda(\theta^{\lambda}_n,)+ \sqrt{2\lambda\beta^{-1}} \xi_{n+1},\quad  n\in\N_0,
\end{equation}
where $\theta_0$ is an $\R^d$-valued random variable, $\lambda>0$ is the stepsize, $(\xi_n)_{n \in \N_0}$ are i.i.d.\ standard $d$-dimensional Gaussian random vectors, and where for each $\lambda>0$, $h_{\lambda}:\R^d \to \R^d$ is a function defined by
\begin{equation}\label{eq:ktulah}
h_\lambda(\theta)\coloneqq a\theta +\frac{h(\theta)-a\theta}{(1+\lambda|\theta|^{ (l+1)/\epsilon_h})^{\epsilon_h}}.
\end{equation}
with $a>0$, $l \in\N$, and $\epsilon_h \in (0,1/2]$. We denote by $\pi^{\lambda}_n$ the density of $\mathcal{L}(\theta^{\lambda}_n)$ for all $n\in\N_0$.

\begin{remark}
The design of the tamed coefficient $h_\lambda$ in \eqref{eq:ktulah} follows from that of mTULA~\cite{neufeld2025non} and sTULA~\cite{lytras2025taming}. It allows us to derive several properties of $h_\lambda$, see Remark \ref{lem:hlambdaestimates} below, which are crucial to establish moment estimates and convergence results of kTULA \eqref{eq:ktula}-\eqref{eq:ktulah}. More precisely, by adopting the splitting trick originally used in \cite{johnston2023kinetic,lytras2025taming}, the tamed coefficient $h_{\lambda}$ \eqref{eq:ktulah} satisfies a dissipativity condition, i.e., Remark~\ref{lem:hlambdaestimates}-\ref{lem:hlambdaestimates1}, which enables a contraction of the algorithm \eqref{eq:ktula} and therefore an easier computation of the associated moment bounds. Moreover, dividing the term $(1+\lambda|\theta|^{ (l+1)/\epsilon_h})^{\epsilon_h}$ in \eqref{eq:ktulah} gives an improved upper bound for the difference between $h_{\lambda}$ and $h$, i.e., Remark~\ref{lem:hlambdaestimates}-\ref{lem:hlambdaestimates4}, compared to the works in \cite{TUSLA,lytras2024tamed,lytras2025taming}, which allows the derivation of improved convergence results in Theorem \ref{thm:mainthm} and Corollary~\ref{crl}.
\end{remark}

\subsection{Assumptions} Let $u \colon \R^d \rightarrow \R$ be three times continuously differentiable. The following assumptions are stated.

We first impose conditions on $\theta_0$ and $\pi^{\lambda}_0$.
\begin{assumption}\label{asm:initialc}
The initial condition $\theta_0$ is independent of $(\xi_{n})_{n\in\N_0}$. Moreover, $\pi^{\lambda}_0$ has exponential decay. $|\nabla \log \pi^{\lambda}_0|$ and $\|\nabla ^2 \log \pi^{\lambda}_0\|$ have polynomial growth.
\end{assumption}

Then, we impose a local Lipschitz condition on $H$.
\begin{assumption}\label{asm:polylip}
There exist $L>0$ and $l \in\N$ such that, for all $\theta, \overline{\theta} \in \R^d$,
\[
\|H(\theta)-H(\overline{\theta})\|\leq L(1+|\theta|+|\overline{\theta}|)^{l-1}|\theta-\overline{\theta}|.
\]
In addition, there exist $K_H, K_h>0$ such that, for all $\theta \in \R^d$,
\[
\|H(\theta)\|\leq K_H(1+|\theta|^{l}), \quad |h(\theta)| \leq K_h(1+|\theta|^{l+1}). 
\]
\end{assumption}

By using Assumption \ref{asm:polylip}, we deduce in the following remark that $h$ satisfies a local Lipschitz condition. The proof is postponed to Appendix \ref{rmk:polyliphproof}.
\begin{remark}\label{rmk:polyliph}
By Assumption \ref{asm:polylip}, for any $\theta, \overline{\theta} \in \R^d$, we obtain the following inequality:
\begin{align*}
|h(\theta) - h(\overline{\theta})| 	&\leq K_H(1+|\theta|+|\overline{\theta}|)^{l}|\theta-\overline{\theta}|.
\end{align*}
\end{remark}

\begin{remark}\label{rmk:dwoptconst}
In Assumption \ref{asm:polylip}, we impose separately growth conditions of $H$ and $h$, which could have been derived by using the local Lipschitz condition on $H$. The reason is that we want to optimize the stepsize restriction $\lambda_{\max}$ defined in \eqref{eq:stepsizemax}, which is related to the reciprocal of $K_H$ and $K_h$. As a concrete example, consider the double-well potential $u(\theta)=(1/4)|\theta|^4-(1/2)|\theta|^2$, $\theta \in \R^d$. Then, the local Lipschitz condition in Assumption \ref{asm:polylip} is satisfied with $L=3$, $l=2$, which implies
\[
\|H(\theta)\|\leq 6(1+|\theta|^{2}), \quad |h(\theta)| \leq 12(1+|\theta|^{3}). 
\] 
However, by the fact that $H(\theta) = (|\theta|^2-1)\cI_d+2\theta\theta^{\cT}$ with $\cI_d$ denoting the $d \times d$ identity matrix and $h(\theta) = \theta(|\theta|^2-1)$, $\theta\in\R^d$, we immediately observe that 
\[
\|H(\theta)\|\leq 3(1+|\theta|^{2}), \quad |h(\theta)| \leq 2(1+|\theta|^{3}). 
\]
Hence, in view of \eqref{eq:stepsizemax}, imposing separately growth conditions of $H$ and $h$ in Assumption \ref{asm:polylip} can help avoid unnecessarily small $\lambda_{\max}$.
\end{remark}

Next, we impose a dissipativity condition on $h$. 
\begin{assumption}\label{asm:dissip}
There exist $a, b>0$ such that, for all $\theta\in\R^d$,
\[
\langle h(\theta),\theta\rangle \geq a|\theta|^2-b.
\]
\end{assumption}

Finally, we impose a Log-Sobolev inequality on $\pi_{\beta}$ defined in \eqref{eq:pibetaexp}.
\begin{assumption}\label{asm:LSI}
$\pi_{\beta}$ satisfies a Log-Sobolev inequality with a constant $C_{\rmLS}>0$.
\end{assumption}

\subsection{Main results} Denote by
\begin{equation}\label{eq:stepsizemax}
\lambda_{\max}\coloneqq \min\{  1,1/(8a), 1/(6\cL_0)^{1/(1-\epsilon_h)}\}
\end{equation}
with $\cL_0 \coloneqq 2a+4K_H+(l+1)(2K_h+a)$ and $\epsilon_h \in (0,1/2]$. 

Recall that $\pi^{\lambda}_n$ denotes the density of $\mathcal{L}(\theta^{\lambda}_n)$ for all $n\in\N_0$. Under Assumptions \ref{asm:initialc}, \ref{asm:polylip}, \ref{asm:dissip}, and \ref{asm:LSI}, we obtain the following non-asymptotic error bound for the KL divergence of $\pi^{\lambda}_n$ w.r.t.\ $\pi_{\beta}$.

\begin{theorem}\label{thm:mainthm}
Let Assumptions \ref{asm:initialc}, \ref{asm:polylip}, \ref{asm:dissip}, and \ref{asm:LSI} hold. Then, for any $\beta>0$, $\epsilon_h \in (0,1/2]$, $\epsilon>0$, there exist positive constants $C_0, C_1$ such that, for any $n \in \N_0$, $0<\lambda\leq \lambda_{\max}$,
\[
\rmKL(\pi^{\lambda}_n\|\pi_{\beta}) \leq e^{-C_0\lambda n}\rmKL(\pi^{\lambda}_0\|\pi_{\beta}) +C_1\lambda^{2-\epsilon_h-\epsilon(1-\epsilon_h/2)},
\]
where $C_0, C_1$ are given explicitly in Appendix \ref{appen:constexp}. Furthermore, for any $\delta>0$, if we choose 
\begin{align*}
\lambda	&\leq \min\{ (\delta/(2C_1))^{1/(2-\epsilon_h-\epsilon(1-\epsilon_h/2))},\lambda_{\max}\},\\
n 		&\geq (1/C_0)\max\{ (2C_1/\delta)^{1/(2-\epsilon_h-\epsilon(1-\epsilon_h/2))},1/\lambda_{\max}\}\log (2\rmKL(\pi^{\lambda}_0\|\pi_{\beta})/\delta),
\end{align*}
then $\rmKL(\pi^{\lambda}_n\|\pi_{\beta})\leq \delta$.
\end{theorem}

Moreover, by using Talagrand’s inequality (see, e.g., \cite[Definition 2.3]{lytras2025taming}), we can further deduce a non-asymptotic convergence bound in Wasserstein-2 distance between $\mathcal{L}(\theta^{\lambda}_n)$ and $\pi_{\beta}$.
\begin{corollary}\label{crl}
Let Assumptions \ref{asm:initialc}, \ref{asm:polylip}, \ref{asm:dissip}, and \ref{asm:LSI} hold. Then, for any $\beta>0$, $\epsilon_h \in (0,1/2]$, $\epsilon>0$, there exist positive constants $C_0, C_1, C_2$ such that, for any $n \in \N_0$, $0<\lambda\leq \lambda_{\max}$,
\[
W_2(\mathcal{L}(\theta^{\lambda}_n),\pi_{\beta}) \leq C_2\left(e^{-C_0\lambda n}\rmKL(\pi^{\lambda}_0\|\pi_{\beta}) +C_1\lambda^{2-\epsilon_h-\epsilon(1-\epsilon_h/2)}\right)^{1/2},
\]
where $C_0, C_1, C_2$ are given explicitly in Appendix \ref{appen:constexp}. Furthermore, for any $\delta>0$, if we choose 
\begin{align*}
\lambda	&\leq \min\{ (\delta/(2C_2\sqrt{C_1}))^{2/(2-\epsilon_h-\epsilon(1-\epsilon_h/2))},\lambda_{\max}\},\\
n 		&\geq (2/C_0)\max\{ (2C_2\sqrt{C_1}/\delta)^{2/(2-\epsilon_h-\epsilon(1-\epsilon_h/2))},1/\lambda_{\max}\}\log (2C_2\rmKL^{1/2}(\pi^{\lambda}_0\|\pi_{\beta})/\delta),
\end{align*}
then $W_2(\mathcal{L}(\theta^{\lambda}_n),\pi_{\beta})\leq \delta$.
\end{corollary}

\begin{remark}
Theorem \ref{thm:mainthm} provides a non-asymptotic convergence bound for the newly proposed kTULA algorithm \eqref{eq:ktula}-\eqref{eq:ktulah} in KL divergence with the rate of convergence equal to $2-\epsilon_h-\epsilon(1-\epsilon_h/2)$, which implies the convergence of kTULA in Wasserstein-2 distance with the rate equal to $1-\epsilon_h/2-\epsilon(1/2-\epsilon_h/4)$ as stated in Corollary \ref{crl}. We note that $\epsilon_h\in(0,1/2]$ arises from the design of the tamed coefficient $h_\lambda$ in \eqref{eq:ktulah}, while $\epsilon>0$ appears due to the application of Young's inequality in the proof (see \eqref{eq:F1ub}). We highlight that both $\epsilon_h\in(0,1/2]$ and $\epsilon>0$ can be chosen arbitrarily, allowing the rates of convergence in KL divergence and in Wasserstein-2 distance to be made arbitrarily close to 2 and 1, respectively. However, achieving higher rates of convergence would result in a worse dependence on the key parameters (including $d$ and $\beta$) of the constant $C_1$. 
\end{remark}

Furthermore, when $\beta>0$ is sufficiently large, $\pi_{\beta}$ defined in \eqref{eq:pibetaexp} concentrates around the minimizers of $u$ \cite{hwang}. Hence, the kTULA algorithm \eqref{eq:ktula}-\eqref{eq:ktulah} can be used to solve the optimization problem:
\[
\text{minimize} \quad \R^d \ni \theta \mapsto u(\theta). 
\]
A non-asymptotic convergence bound for the associated expected excess risk can be deduced using the result in Corollary \ref{crl} as presented below.

\begin{corollary}\label{crl:eer}
Let Assumptions \ref{asm:initialc}, \ref{asm:polylip}, \ref{asm:dissip}, and \ref{asm:LSI} hold. Then, for any $\beta>0$, $\epsilon_h \in (0,1/2]$, $\epsilon>0$, there exist positive constants $C_3, C_4, C_5$ such that, for any $n \in \N_0$, $0<\lambda\leq \lambda_{\max}$,
\[
\E[u( \theta_n^{\lambda})] - \inf_{\theta \in \R^d} u(\theta) \leq C_3 e^{-C_0\lambda n/2}+C_4\lambda^{1-\epsilon_h/2-\epsilon(1/2-\epsilon_h/4)}+C_5/\beta,
\]
where $C_3, C_4, C_5$ are given explicitly in Appendix \ref{appen:constexp}. Furthermore, for any $\delta>0$, if we choose 
\begin{align*}
\beta		&\geq \max\left\{1,\frac{9d^2}{\delta^2}, \frac{3d}{\delta}\log\left(\frac{K_H(1+4(\sqrt{b/a}+\sqrt{2d/K_H}))^l(b+1)(d+1)}{ad}\right)+\frac{6\log 2}{\delta}\right\},\\
\lambda	&\leq \min\{ (\delta/(3C_4))^{2/(2-\epsilon_h-\epsilon(1-\epsilon_h/2))},\lambda_{\max}\},\\
n 		&\geq (2/C_0)\max\{ (3C_4/\delta)^{2/(2-\epsilon_h-\epsilon(1-\epsilon_h/2))},1/\lambda_{\max}\}\log (3C_3/\delta),
\end{align*}
then $\E[u( \theta_n^{\lambda})] - \inf_{\theta \in \R^d} u(\theta) \leq \delta$.
\end{corollary}

The proofs of Theorem \ref{thm:mainthm}, Corollary \ref{crl}, and Corollary \ref{crl:eer} are deferred in Section \ref{sec:mtproofs}.

\subsection{Related work and comparison}\label{sec:literaturereview}
In Theorem \ref{thm:mainthm} and Corollary \ref{crl}, we establish, under Assumptions~\ref{asm:initialc}-\ref{asm:LSI}, non-asymptotic convergence bounds in KL divergence and in Wasserstein distance between the law of kTULA \eqref{eq:ktula}-\eqref{eq:ktulah} and $\pi_{\beta}$ \eqref{eq:pibetaexp}, respectively. In particular, we show that the rate of convergence in KL divergence is $2-\epsilon_h-\epsilon(1-\epsilon_h/2)$, where $\epsilon_h\in(0,1/2]$ and $\epsilon>0$ can be chosen arbitrarily, while the rate in Wasserstein-2 distance is $1-\epsilon_h/2-\epsilon(1/2-\epsilon_h/4)$. In this section, we compare the aforementioned results with those obtained in \cite{neufeld2025non}, \cite{lytras2025taming}, and \cite{bao2024}.

In \cite{neufeld2025non}, the authors proposed the so-called modified tamed unadjusted Langevin algorithm (mTULA) to sample from a high-dimensional target distribution $\pi_{\beta}$. \cite[Theorem 2.9 and Theorem 2.10]{neufeld2025non} provide non-asymptotic error bounds in Wasserstein-1 and Wasserstein-2 distances between the law of mTULA and $\pi_{\beta}$ with rates of convergence being $1$ and $1/2$, respectively. These results are obtained under the conditions that the initial value of mTULA has a finite $r_*$-th moment with $r_*>0$ \cite[Assumption 1]{neufeld2025non}, that $h$ satisfies a polynomial Lipschitz condition \cite[Assumption 2]{neufeld2025non} and a convexity at infinity condition \cite[Assumption 3]{neufeld2025non}, and that $H$ satisfies a polynomial Lipschitz condition \cite[Assumption 4]{neufeld2025non}. Our Corollary~\ref{crl} improves the rate of convergence in Wasserstein-2 distance obtained in \cite[Theorem 2.10]{neufeld2025non} from $1/2$ to $1-\epsilon_h/2-\epsilon(1/2-\epsilon_h/4)$, which, e.g., equals to $7/10$ if we set $\epsilon_h = 1/2$ and $\epsilon= 2/15$. One can further improve the rate to a value which is arbitrarily close to 1 but with a worse dependence of $C_1$ on the dimension. This improvement is achieved under our Assumptions~\ref{asm:initialc}-\ref{asm:LSI}. While Assumption~\ref{asm:initialc} imposes additional conditions on $\mathcal{L}(\theta_0)$ and hence is stronger than \cite[Assumption 1]{neufeld2025non}, Assumptions~\ref{asm:dissip} and \ref{asm:LSI} are weaker conditions compared to \cite[Assumption 3]{neufeld2025non}. In addition, Assumption~\ref{asm:polylip} is similar to \cite[Assumption 4]{neufeld2025non}.

Next, we compare with \cite{lytras2025taming}. \cite[Eq.\ (3) and (4)]{lytras2025taming} proposes the splitted tamed unadjusted Langevin algorithm (sTULA). \cite[Theorem 5.5 and Corollary 5.7]{lytras2025taming} provide convergence results for sTULA in KL divergence and in Wasserstein-2 distance with the rates of convergence being $1$ and $1/2$, respectively. These results are obtained under \cite[A1-A4, and B1]{lytras2025taming}. More precisely, \cite[A1]{lytras2025taming} imposes growth conditions on (the derivatives) of $h$ while \cite[A2]{lytras2025taming} imposes a polynomial Lipschitz condition on $h$. \cite[A3]{lytras2025taming} imposes a dissipativity condition on $h$ and \cite[A4]{lytras2025taming} imposes certain conditions on the initial value of sTULA. \cite[B1]{lytras2025taming} assumes that $\pi_\beta$ satisfies a log-Sobolev inequality. Our Theorem \ref{thm:mainthm} and Corollary \ref{crl} improve the rates of convergence obtained in \cite[Theorem 5.5 and Corollary 5.7]{lytras2025taming} under a weaker condition (Assumption~\ref{asm:polylip}) compared to \cite[A1]{lytras2025taming}, as the polynomial Lipschitz condition in Assumption~\ref{asm:polylip} can be implied by \cite[A1]{lytras2025taming}. We note that our Assumptions~\ref{asm:initialc}, \ref{asm:dissip}, and \ref{asm:LSI} are the same as \cite[A4, A3, and B1]{lytras2025taming}.

Finally, we compare with \cite{bao2024}. \cite[Theorem 1.9]{bao2024} provides a non-asymptotic convergence result in Wasserstein-2 distance for the tamed algorithm \cite[Eq. (1.24)]{bao2024} with the rate of convergence equal to 1. This result is established under local Lipschitz conditions of $h$ and $\nabla h$ \cite[Eq. (1.11) and {H6}]{bao2024}, a convexity at infinity condition \cite[Eq. (1.19)]{bao2024}, and the condition that $\beta>0$ is sufficiently small. We note that \cite[Theorem 1.9]{bao2024} can also be obtained for our kTULA algorithm \eqref{eq:ktula}-\eqref{eq:ktulah}, which improves the rate of convergence obtained in Corollary \ref{crl} from $1-\epsilon_h/2-\epsilon(1/2-\epsilon_h/4)$ (which is arbitrarily close to 1) to exact 1. However, this improvement is achieved under stronger assumptions on $h$ and $\beta$. Namely, it requires a convexity at infinity condition on $h$ \cite[Eq. (1.19)]{bao2024} which would imply our Assumptions~\ref{asm:dissip} and \ref{asm:LSI}, and hence a stronger condition. In addition, it requires $\beta>0$ to be sufficiently small while we require only $\beta>0$. Hence, \cite[Theorem 1.9]{bao2024} cannot be applied to optimization problems where $\beta>0$ usually takes large values, whereas our Corollary \ref{crl} can be used to provide a theoretical guarantee for kTULA~\eqref{eq:ktula}-\eqref{eq:ktulah} to solve optimization problems as shown in Corollary \ref{crl:eer}. Regarding the other assumptions, \cite[Theorem 1.9]{bao2024} requires the initial value $\theta_0$ to have finite polynomial moments which is weaker than our Assumption~\ref{asm:initialc}, however, the conditions in Assumption~\ref{asm:initialc} can be easily satisfied by, e.g., constants and Gaussian random variables that are typically used for initializing algorithms. Moreover, \cite[Eq. (1.11) and {H6}]{bao2024} are conditions comparable to our Assumption~\ref{asm:polylip}.

\section{Applications}\label{sec:app}
In this section, we apply the kTULA algorithm \eqref{eq:ktula}-\eqref{eq:ktulah} to solve the problem of sampling from a double-well distribution and an optimization problem involving a neural network. This demonstrates the wide applicability of our main results. 
\subsection{Sampling from a double-well potential distribution}
We consider the problem of sampling from a target distribution $\pi_{\beta}$ of the form \eqref{eq:pibetaexp} with the potential given by 
\begin{equation}\label{eq:dwpotential}
u(\theta)=|\theta|^4/4-|\theta|^2/2,\quad \theta\in\R^d.
\end{equation}

Our goal is to sample from $\pi_{\beta}$ using the kTULA algorithm \eqref{eq:ktula}-\eqref{eq:ktulah}. We can show that $\pi_{\beta}$ satisfies our Assumptions \ref{asm:polylip}-\ref{asm:LSI}. If we further choose the initial value $\theta_0$ to be, e.g., a constant or a Gaussian random variable, then Assumption \ref{asm:initialc} is satisfied, and hence Corollary \ref{crl} can be applied to provide a theoretical guarantee for the sampling behavior of kTULA.

\begin{proposition}\label{prop:dw} The target distribution $\pi_{\beta}$ with the potential $u$ defined in \eqref{eq:dwpotential} satisfies Assumptions \ref{asm:polylip}-\ref{asm:LSI}.
\end{proposition}
\begin{proof}See Appendix \ref{prop:dwproof}.
\end{proof}


\subsection{Optimization involving a neural network in the transfer learning setting} We consider the following optimization problem:
\begin{equation}\label{eq:optobj}
\text{minimize} \quad \R^d \ni \theta \mapsto u(\theta) \coloneqq \E\left[(Y-\mathfrak{N}(\theta,Z))^2\right] + (\eta/6) |\theta|^6,
\end{equation}
where $X=(Z,Y)\in\R^m$ is the data point with $Y$ being an $\R$-valued target random variable and $Z$ being an $\R^{m-1}$-valued input random variable, $\eta>0$ is a regularization constant, and where $\mathfrak{N}:\R^d\times\R^{m-1}\to\R$ is a single-hidden-layer feed-forward neural network (1LFN) given by
\begin{equation}\label{eq:1lfn}
\mathfrak{N}(\theta,z) \coloneqq \sum_{i=1}^{d_1}W_1^{1i}\sigma_0(\langle c_0^{i\cdot}, z\rangle+b_0^i)
\end{equation}
with $z\in\R^{m-1}$ being the input vector, $W_1\in\R^{1\times d_1}$ being the weight parameter, $c_0\in\R^{d_1\times(m-1)}$ being the fixed (pretrained nonzero) input weight, $b_0\in\R^{d_1}$ being the bias parameter, and $\sigma_0(x) = x/(1+e^{-x}), x\in\R$, being the sigmoid linear unit activation function \cite{atto2008smooth}. We note that $\theta = ([W_1],b_0)\in\R^d$, $d=2d_1$, is the parameter of the optimization problem, where $[A]$ denotes the vector of all elements in a given matrix $A$.

We aim to solve \eqref{eq:optobj} using the kTULA algorithm \eqref{eq:ktula}-\eqref{eq:ktulah}. By showing that the objective function $u$ of \eqref{eq:optobj} satisfies Assumptions \ref{asm:polylip}-\ref{asm:LSI} together with an appropriately chosen $\theta_0$ that satisfies Assumption \ref{asm:initialc}, we can use Corollary \ref{crl:eer} to provide a theoretical guarantee for kTULA to obtain approximate minimzers of $u$ in \eqref{eq:optobj}.

\begin{proposition}\label{prop:opt} The objective function $u$ defined in \eqref{eq:optobj} satisfies Assumptions \ref{asm:polylip}-\ref{asm:LSI}.
\end{proposition}
\begin{proof}See Appendix \ref{prop:optproof}.
\end{proof}

\section{Proof of main results}\label{sec:mtproofs}
In this section, we provide the proofs for Theorem \ref{thm:mainthm}, Corollary \ref{crl}, and Corollary \ref{crl:eer}. We first provide an overview of our proof techniques for Theorem \ref{thm:mainthm}. Then, we proceed to introduce a couple of (auxiliary) stochastic processes to facilitate the convergence analysis. Next, we establish moment estimates for kTULA \eqref{eq:ktula}-\eqref{eq:ktulah}. Finally, we provide the detailed proofs of our main results. All the proofs for the results presented in this section are postponed to Appendix \ref{sec:mtproofsapdx}.

\subsection{Overview of proof techniques for Theorem \ref{thm:mainthm}} 
Recall $h_{\lambda}$ defined in \eqref{eq:ktulah}. Our kTULA algorithm~\eqref{eq:ktula}-\eqref{eq:ktulah} possesses the following essential properties:
\begin{enumerate}
[label=\textbf{\upshape(P\arabic*)}]
\setcounter{enumi}{2}
    \item $h_{\lambda}$ satisfies a dissipativity property, i.e., Assumption \ref{asm:dissip}, see Lemma \ref{lem:hlambdaestimates}-\ref{lem:hlambdaestimates1}, \label{enum:me1}
    \item $h_{\lambda}$ is growing linearly, see Lemma \ref{lem:hlambdaestimates}-\ref{lem:hlambdaestimates2}, \label{enum:me2}
    \item $h_{\lambda}$ satisfies a global Lipschitz condition with the Lipschitz constant being $\mathcal{O}( \lambda^{-\epsilon_h})$, $\epsilon_h\in(0,1/2]$, see Lemma~\ref{lem:hlambdaestimates}-\ref{lem:hlambdaestimates3},  \label{enum:ca1}
    \item The approximation of $h_\lambda$ to $h$ is of order $\lambda$, see Lemma \ref{lem:hlambdaestimates}-\ref{lem:hlambdaestimates4}.  \label{enum:ca2}
\end{enumerate}
The first two properties \ref{enum:me1} and \ref{enum:me2} enable us to obtain uniform in time moment estimates of kTULA (see Lemma~\ref{lem:2ndpthmmt}), which also enable a rigorous derivation of the functional equality \eqref{eq:changeofKLint} in the spirit of \cite[Section 9]{lytras2025taming} (see Appendix \ref{appen:changeintnder} for more details), whereas the last two properties \ref{enum:ca1} and \ref{enum:ca2} are used in the convergence analysis. 

To establish a non-asymptotic convergence bound for kTULA in KL divergence with an improved rate of convergence $2-\overline{\epsilon}$, $\overline{\epsilon}>0$, we start from \eqref{eq:changeofKLint}, which can be obtained by using the same proof as that of \cite[Theorem 7.8]{lytras2025taming}, and adopt the decomposition proposed in \cite{mou2022improved} to obtain \eqref{eq:decomp} and \eqref{eq:decompterms}. Our task then reduces to upper bound each of the terms on the RHS of \eqref{eq:decomp}. We note that the upper estimates for (the expectations of) the second to the fourth terms (on the RHS of \eqref{eq:decomp}) can be obtained by applying standard techniques as shown in the proof of Lemma~\ref{lem:decomptermsF2n3} and \ref{lem:decomptermsrt}, and their upper bounds are all proved to be $C\lambda^2$ with $C>0$. To obtain an upper estimate for the first term on the RHS of \eqref{eq:decomp}, we provide a novel proof that extends the approach adopted in \cite{mou2022improved} to the case of super-linearly growing $h$. More precisely, we establish in Lemma \ref{lem:ubforfisherinfo} a uniform upper estimate for the Fisher information $J_n$ of the distribution of the $n$-th iterate of kTULA (denoted by $\pi^{\lambda}_n$). This can then be used to obtain an upper bound for the first term on the RHS of \eqref{eq:decomp} as presented in Lemma \ref{lem:decomptermsF1}, which is of the form $C\lambda^{2-\overline{\epsilon}}$ with $C,\overline{\epsilon}>0$. We highlight that all these results are achieved due to the design of kTULA in \eqref{eq:ktula} and \eqref{eq:ktulah} possessing the properties \ref{enum:me1}-\ref{enum:ca2}. Finally, we use all these estimates to obtain an upper bound for the second term on the RHS \eqref{eq:changeofKLint} via \eqref{eq:decomp}, and the application of Assumption \ref{asm:LSI} yields the desired result.

\subsection{Auxiliary processes} Fix $\beta>0$. Consider the Langevin SDE $(Z_t)_{t \geq 0}$ given by
\begin{equation} \label{eq:sde}
Z_0 \coloneqq  \theta_0, \quad \rmd Z_t=-h\left(Z_t\right) \rmd t+ \sqrt{2\beta^{-1}} \rmd B_t,
\end{equation}
where $(B_t)_{t \geq 0}$ is a $d$-dimensional Brownian motion with its completed natural filtration denoted by $(\mathcal{F}_t)_{t\geq 0}$. We assume that $(\mathcal{F}_t)_{t\geq 0}$ is independent of $\sigma(\theta_0)$. Moreover, it is a well-known result that, by Remark \ref{rmk:polyliph} (under Assumption \ref{asm:polylip}) and by Assumption \ref{asm:dissip}, the Langevin SDE \eqref{eq:sde} admits a unique strong solution, which is adapted to $\mathcal{F}_t \vee \sigma(\theta_0)$, $t\geq 0$, see, e.g., \cite[Theorem 1]{krylovsolmontone}. The $2p$-th moment estimate of the Langevin SDE \eqref{eq:sde} with $p \in \N$ has been established in \cite[Lemma A.1]{lim2021nonasymptotic}, which can be used to deduce the $2p$-th moment estimate of $\pi_{\beta}$ \eqref{eq:pibetaexp}.

For each $\lambda>0$, define $ B^{\lambda}_t \coloneqq B_{\lambda t}/\sqrt{\lambda}, t \geq 0$. We note that $( B^{\lambda}_t )_{t\geq 0}$ is a $d$-dimensional standard Brownian motion. Its completed natural filtration is denoted by $(\mathcal{F}^{\lambda}_t)_{t \geq 0}$ with $\mathcal{F}^{\lambda}_t\coloneqq \mathcal{F}_{\lambda t}$ for each $t\geq 0$, which is also independent of $\mathcal{G}_{\infty} \vee \sigma(\theta_0)$. Then, 
we denote by $(\overline{\theta}^{\lambda}_t)_{t \geq 0 }$ the continuous-time interpolation of kTULA \eqref{eq:ktula}-\eqref{eq:ktulah} given by
\begin{equation}\label{eq:ktulaproc}
\rmd \overline{\theta}^{\lambda}_t=-\lambda h_{\lambda}(\overline{\theta}^{\lambda}_{\lfrf{t}})\, \rmd t
+ \sqrt{2\lambda\beta^{-1}} \rmd B^{\lambda}_t
\end{equation}
with the initial condition $\bar{\theta}^{\lambda}_0\coloneqq \theta_0$. We note that the law of the interpolated process \eqref{eq:ktulaproc} coincides with the law of kTULA \eqref{eq:ktula}-\eqref{eq:ktulah} at grid points, i.e., $\mathcal{L}(\overline{\theta}^{\lambda}_n) =\mathcal{L}(\theta^{\lambda}_n) $, for all $n\in\N_0$. We denote by $\pi^{\lambda}_t$ the density of $\mathcal{L}(\overline{\theta}^{\lambda}_t) $ for all $n\in\N_0$, $t\in[n,n+1]$.

\subsection{Moment estimates}
Recall $h_{\lambda}$ defined in \eqref{eq:ktulah}. We first investigate the properties of $h_{\lambda}$ which are useful in our analysis. The results are presented in the following lemma.
\begin{lemma}\label{lem:hlambdaestimates} Let Assumptions \ref{asm:polylip} and \ref{asm:dissip} hold. Then, we obtain the following:
\begin{enumerate}
\item\label{lem:hlambdaestimates1} For any $\theta\in\R^d$, $0<\lambda <1$, and $\epsilon_h \in (0,1/2]$,
\[
\langle h_{\lambda}(\theta),\theta\rangle\geq a |\theta|^2-b.
\] 
\item\label{lem:hlambdaestimates2} For any $\theta\in\R^d$, $0<\lambda <1$, and $\epsilon_h \in (0,1/2]$,
\[
|h_{\lambda}(\theta)|\leq 2a |\theta|+ 2K_h \lambda^{-1/2} , \quad |h_{\lambda}(\theta)|\leq (2a+K_h)(1+ |\theta|^{l+1}).
\] 
\item\label{lem:hlambdaestimates3} For any $\theta, \overline{\theta}\in\R^d$, $0<\lambda <1$, and $\epsilon_h \in (0,1/2]$,
\begin{align*}
| h_{\lambda}(\theta)-  h_{\lambda}(\overline{\theta})| 			&\leq \cL_0\lambda^{-\epsilon_h}|\theta-\overline{\theta}|,\\
\|\nabla h_{\lambda}(\theta)-\nabla h_{\lambda}(\overline{\theta})\|	&\leq \cL_{\nabla,\epsilon_h} (1+|\theta|+|\overline{\theta}|)^{(l+1)(1/\epsilon_h+1)-2}|\theta-\overline{\theta}|,
\end{align*}
where $\cL_0 \coloneqq 2a+4K_H+(l+1)(2K_h+a)$ and $\cL_{\nabla,\epsilon_h}\coloneqq (10\sqrt{2}+4\epsilon_h^{-1})(l+1)^2\max\{K_H,L,\\K_h,a\}$. 
\item\label{lem:hlambdaestimates4} For any $\theta\in\R^d$, $0<\lambda <1$, and $\epsilon_h \in (0,1/2]$,
\[
|h(\theta)-h_{\lambda}(\theta)|^2\leq 4 \lambda^2(K_h+a)^2(1+ |\theta|^{2(l+1)(1+1/\epsilon_h)}).
\] 
\end{enumerate}
\end{lemma}
\begin{proof} See Appendix \ref{lem:hlambdaestimatesproof}.
\end{proof}

Then, we establish moment estimates for $(\overline{\theta}^{\lambda}_t)_{t \geq 0 }$ defined in \eqref{eq:ktulaproc}.
\begin{lemma}\label{lem:2ndpthmmt} Let Assumptions \ref{asm:initialc}, \ref{asm:polylip}, and \ref{asm:dissip} hold.  Then, we obtain the following estimates:
\begin{enumerate} 
\item For any $0<\lambda\leq \lambda_{\max}$, $n \in \N_0$, and $t \in (n, n+1]$,
\[
\E\left[ |\overline{\theta}^{\lambda}_t|^2  \right]  \leq (1-a\lambda (t-n))(1-a\lambda)^n\E\left[|\theta_0|^2\right]+\cc_0(1+1/a),
\]
where $\cc_0\coloneq 2b+8K_h^2+2\beta^{-1}d$. In particular, the above inequality implies $\sup_{t\geq 0}\E\left[|\overline{\theta}^{\lambda}_t|^2\right]  \leq  \E\left[|\theta_0|^2\right] +\cc_0(1+1/a)<\infty$.
\label{lem:2ndpthmmti}
\item  For any $p\in [2, \infty)\cap {\N}$, $0<\lambda\leq \lambda_{\max}$, $n \in \N_0$, and $t \in (n, n+1]$,
\[
\E\left[ |\overline{\theta}^{\lambda}_t|^{2p} \right]  \leq  (1-a\lambda (t-n))  (1-a\lambda)^n\E\left[|\theta_0|^{2p} \right]+ \cc_p(1+1/a),
\]
where $\cc_p$ is given in \eqref{eq:2pthmmtexpconst}. In particular, the above estimate implies $\sup_{t\geq 0}\E\left[|\overline{\theta}^{\lambda}_t|^{2p} \right]  \leq  \E\left[|\theta_0|^{2p}\right]+ \cc_p(1+1/a)<\infty$. \label{lem:2ndpthmmtii}
\end{enumerate}
\end{lemma}
\begin{proof} See Appendix \ref{lem:2ndpthmmtproof}.
\end{proof}

Next, we provide a uniform upper bound for the density of $\mathcal{L}(\overline{\theta}^{\lambda}_t)$, $t\in(n,n+1], n \in \N_0$.
\begin{lemma}{\cite[Proposition 2]{polyanskiy2016wasserstein}}\label{lem:ktuladensitygrowth} Let Assumptions \ref{asm:initialc}, \ref{asm:polylip}, and \ref{asm:dissip} hold.  Then, we have, for any $\theta\in\R^d$, $0<\lambda\leq \lambda_{\max}$, $n \in \N_0$, that
\[
|\nabla \log \pi^{\lambda}_n(\theta)| \leq \lambda^{-1}(\widetilde{\cc}_1|\theta|+\widetilde{\cc}_2),
\]
where $\widetilde{\cc}_1 \coloneqq 2\beta  $ and $\widetilde{\cc}_2 \coloneqq 2\beta  ( \E\left[|\theta_0|^2\right] +\cc_0(1+1/a))^{1/2}$ with $\cc_0$ given in Lemma \ref{lem:2ndpthmmt}-\ref{lem:2ndpthmmti}.
\end{lemma}

\begin{proof} See Appendix \ref{lem:ktuladensitygrowthproof}.
\end{proof}

\subsection{Proof of main results}
To establish the results in Theorem \ref{thm:mainthm} and Corollary \ref{crl}, we follow the proof of \cite[Theorem 7.8]{lytras2025taming} to obtain, for any $0<\lambda\leq \lambda_{\max}$, $n\in\N_0$, $t\in(n ,n+1 ]$, that
\begin{equation}\label{eq:changeofKLint}
\frac{\rmd}{\rmd t}\rmKL(\pi^{\lambda}_t\|\pi_{\beta}) = -(3/4)\lambda J(\pi^{\lambda}_t\|\pi_{\beta})+4\lambda\beta\int_{\R^d}\pi^{\lambda}_t(\theta)\left|\E\left[\left. h_{\lambda}(\overline{\theta}^{\lambda}_n) - h(\overline{\theta}^{\lambda}_t)\right|\overline{\theta}^{\lambda}_t=\theta\right]\right|^2\,\rmd\theta,
\end{equation}
where $J(\pi^{\lambda}_t\|\pi_{\beta})\coloneqq (1/\beta)\int_{\R^d}|\nabla \log \pi^{\lambda}_t - \nabla \log \pi_{\beta}|^2\,\rmd \pi_{\beta}$. To upper bound the second term on the RHS of \eqref{eq:changeofKLint}, we consider the decomposition adopted in \cite{mou2022improved}. More precisely, we have, for any $\theta\in\R^d$, that
\begin{equation}\label{eq:decomp}
\E\left[\left. h_{\lambda}(\overline{\theta}^{\lambda}_n) - h(\overline{\theta}^{\lambda}_t)\right|\overline{\theta}^{\lambda}_t=\theta\right] 
=\nabla h_{\lambda}(\theta)(F_1(\theta)-F_2(\theta)-F_3(\theta) )+r_t(\theta)+\overline{r}_t(\theta),
\end{equation}
where
\begin{align}\label{eq:decompterms}
\begin{split}
F_1(\theta)			&\coloneqq \E\left[\left.  \left(\cI_d+\lambda(t-n ) (-\nabla h_{\lambda}(\overline{\theta}^{\lambda}_n))\right)\left(\overline{\theta}^{\lambda}_n-\overline{\theta}^{\lambda}_t+\lambda(t-n )(-  h_{\lambda}(\overline{\theta}^{\lambda}_n))\right) \right|\overline{\theta}^{\lambda}_t=\theta\right],\\
F_2(\theta)			&\coloneqq \E\left[\left.  \lambda(t-n ) (-\nabla h_{\lambda}(\overline{\theta}^{\lambda}_n))\left(\overline{\theta}^{\lambda}_n-\overline{\theta}^{\lambda}_t+\lambda(t-n )(-  h_{\lambda}(\overline{\theta}^{\lambda}_n))\right) \right|\overline{\theta}^{\lambda}_t=\theta\right],\\
F_3(\theta)			&\coloneqq \E\left[\left.  \lambda(t-n )(-  h_{\lambda}(\overline{\theta}^{\lambda}_n)) \right|\overline{\theta}^{\lambda}_t=\theta\right],\\
r_t(\theta)			&\coloneqq \E\left[\left. h_{\lambda}(\overline{\theta}^{\lambda}_n) - h_{\lambda}(\overline{\theta}^{\lambda}_t)\right|\overline{\theta}^{\lambda}_t=\theta\right] - \nabla h_{\lambda}(\theta) \E\left[\left. \overline{\theta}^{\lambda}_n-\overline{\theta}^{\lambda}_t\right|\overline{\theta}^{\lambda}_t=\theta\right] ,\\
\overline{r}_t(\theta)	&\coloneqq \E\left[\left. h_{\lambda}(\overline{\theta}^{\lambda}_t) - h(\overline{\theta}^{\lambda}_t)\right|\overline{\theta}^{\lambda}_t=\theta\right] .
\end{split}
\end{align}
We then establish upper bound for each of the terms in \eqref{eq:decompterms}.

We define, for any $n\in\N_0$, 
\begin{equation}\label{eq:fisherinfo}
J_n \equiv J(\pi^{\lambda}_n)\coloneqq \int_{\R^d} \pi^{\lambda}_n(\theta) |\nabla \log \pi^{\lambda}_n(\theta)|^2\,\rmd \theta.
\end{equation}
The upper bound associated with the first term in \eqref{eq:decompterms} is provided in the following lemma.
\begin{lemma}\label{lem:decomptermsF1} Let Assumptions \ref{asm:initialc}, \ref{asm:polylip}, and \ref{asm:dissip} hold. Let $\epsilon_h \in (0,1/2]$ and $\epsilon>0$. Then, we have, for any $0<\lambda\leq \lambda_{\max}$, $n\in\N_0$, $t\in(n ,n+1 ]$, that
\[
\E \left[\left| \nabla h_{\lambda}(\overline{\theta}^{\lambda}_t) F_1(\overline{\theta}^{\lambda}_t)\right|^2\right] \leq \cC_{D,\epsilon,\epsilon_h}\lambda^{2-\epsilon} J_n^{1-\epsilon/2},
\]
where $\cC_{D,\epsilon,\epsilon_h} >0$ is given in \eqref{eq:cDepsilonexp}.
\end{lemma}
\begin{proof} See Appendix \ref{lem:decomptermsF1proof}.
\end{proof}

The RHS of the inequality in Lemma \ref{lem:decomptermsF1} depends on the term $J_n$, which can be further bounded using the result presented below. Consequently, we provide  a full estimate for the first term in \eqref{eq:decompterms}.
\begin{lemma}\label{lem:ubforfisherinfo} Let Assumptions \ref{asm:initialc}, \ref{asm:polylip}, and \ref{asm:dissip} hold. Let $\epsilon_h \in (0,1/2]$. Then, we have, for any $0<\lambda\leq \lambda_{\max}$, $n\in\N_0$, that
\[
J_n \leq \cC_{J,\epsilon_h}	 \lambda^{-\epsilon_h},
\]
where $\cC_{J,\epsilon_h}>0$ is given in \eqref{eq:cJexp}.
\end{lemma}
\begin{proof} See Appendix \ref{lem:ubforfisherinfoproof}.
\end{proof}

Then, we continue to establish an upper bound associated with the second and third terms in \eqref{eq:decompterms}.
\begin{lemma}\label{lem:decomptermsF2n3} Let Assumptions \ref{asm:initialc}, \ref{asm:polylip}, and \ref{asm:dissip} hold. Let $\epsilon_h \in (0,1/2]$. Then, we have, for any $0<\lambda\leq \lambda_{\max}$, $n\in\N_0$, $t\in(n ,n+1 ]$, that
\[
\E \left[\left| \nabla h_{\lambda}(\overline{\theta}^{\lambda}_t) F_2(\overline{\theta}^{\lambda}_t)\right|^2\right] + \E \left[\left| \nabla h_{\lambda}(\overline{\theta}^{\lambda}_t) F_3(\overline{\theta}^{\lambda}_t)\right|^2\right]\leq \cC_{D,\epsilon_h} \lambda^{2},
\]
where $\cC_{D,\epsilon_h} >0$ is given in \eqref{eq:cDexp}.
\end{lemma}
\begin{proof}See Appendix \ref{lem:decomptermsF2n3proof}.
\end{proof}

Finally, we present an upper estimate associated with the last two terms in \eqref{eq:decompterms} in the following lemma.
\begin{lemma}\label{lem:decomptermsrt} Let Assumptions \ref{asm:initialc}, \ref{asm:polylip}, and \ref{asm:dissip} hold. Let $\epsilon_h \in (0,1/2]$. Then, we have, for any $0<\lambda\leq \lambda_{\max}$, $n\in\N_0$, $t\in(n ,n+1 ]$, that
\[
\E \left[| r_t(\overline{\theta}^{\lambda}_t)  |^2\right] +\E \left[| \overline{r}_t(\overline{\theta}^{\lambda}_t)  |^2\right] \leq \cC_{D,\epsilon_h} \lambda^2,
\]
where $\cC_{D,\epsilon_h}>0$ is given in \eqref{eq:cDexp}.
\end{lemma}
\begin{proof}See Appendix \ref{lem:decomptermsrtproof}.
\end{proof}

\begin{proof}[\textbf{Proof of Theorem \ref{thm:mainthm}}]
Let $\epsilon_h \in (0,1/2]$ and $\epsilon>0$. By using \eqref{eq:changeofKLint}, \eqref{eq:decompterms}, Lemma~\ref{lem:decomptermsF1}, Lemma~\ref{lem:ubforfisherinfo}, Lemma~\ref{lem:decomptermsF2n3}, and Lemma~\ref{lem:decomptermsrt}, we obtain for any $0<\lambda\leq \lambda_{\max}$, $n\in\N_0$, $t\in(n ,n+1 ]$, that
\begin{align*}
\frac{\rmd}{\rmd t}\rmKL(\pi^{\lambda}_t\|\pi_{\beta}) 
&\leq   -(3/4)\lambda J(\pi^{\lambda}_t\|\pi_{\beta})+20\beta\lambda^{3-\epsilon_h-\epsilon(1-\epsilon_h/2)}(\cC_{D,\epsilon,\epsilon_h}\cC_{J,\epsilon_h}^{1-\epsilon/2}+2\cC_{D,\epsilon_h} )\\
&\leq  -(3C_{\rmLS}/2)\lambda \rmKL(\pi^{\lambda}_t\|\pi_{\beta})+20\beta\lambda^{3-\epsilon_h-\epsilon(1-\epsilon_h/2)}(\cC_{D,\epsilon,\epsilon_h}\cC_{J,\epsilon_h}^{1-\epsilon/2}+2\cC_{D,\epsilon_h} ),
\end{align*}
where the second inequality holds due to Assumption \ref{asm:LSI} (see also \cite[Definition 2.2]{lytras2025taming}). Finally, straightforward calculations yields
\[
\rmKL(\pi^{\lambda}_t\|\pi_{\beta}) \leq e^{-C_0\lambda t}\rmKL(\pi^{\lambda}_0\|\pi_{\beta}) +C_1\lambda^{2-\epsilon_h-\epsilon(1-\epsilon_h/2)},
\]
where $C_0	 \coloneqq 3C_{\rmLS}/2$ and $ C_1	 \coloneqq (40\beta/(3C_{\rmLS})) (\cC_{D,\epsilon,\epsilon_h}\cC_{J,\epsilon_h}^{1-\epsilon/2}+2\cC_{D,\epsilon_h} )$ with $\cC_{D,\epsilon,\epsilon_h}, \cC_{J,\epsilon_h}, \cC_{D,\epsilon_h} >0$ given in \eqref{eq:cDepsilonexp}, \eqref{eq:cJexp}, and \eqref{eq:cDexp}, respectively. This implies, for any $0<\lambda\leq \lambda_{\max}$, $n\in\N_0$, that
\begin{equation}\label{eq:KLub}
\rmKL(\pi^{\lambda}_n\|\pi_{\beta}) \leq e^{-C_0\lambda n}\rmKL(\pi^{\lambda}_0\|\pi_{\beta}) +C_1\lambda^{2-\epsilon_h-\epsilon(1-\epsilon_h/2)}.
\end{equation}
which completes the proof.
\end{proof}

\begin{proof}[\textbf{Proof of Corollary \ref{crl}}]
Let $\epsilon_h \in (0,1/2]$ and $\epsilon>0$. By Assumption \ref{asm:LSI}, \eqref{eq:KLub} (in the proof of Theorem~\ref{thm:mainthm}), and by using Talagrand inequality, we obtain, for any $0<\lambda\leq \lambda_{\max}$, $n\in\N_0$, that
\[
W_2(\mathcal{L}(\theta^{\lambda}_n),\pi_{\beta}) \leq C_2\left(e^{-C_0\lambda n}\rmKL(\pi^{\lambda}_0\|\pi_{\beta}) +C_1\lambda^{2-\epsilon_h-\epsilon(1-\epsilon_h/2)}\right)^{1/2},
\]
where $C_0	 \coloneqq 3C_{\rmLS}/2$, $ C_1	 \coloneqq (40\beta/(3C_{\rmLS})) (\cC_{D,\epsilon,\epsilon_h}\cC_{J,\epsilon_h}^{1-\epsilon/2}+2\cC_{D,\epsilon_h} )$, and $C_2 \coloneqq (2C_{\rmLS})^{1/2}$.
\end{proof}

\begin{proof}[\textbf{Proof of Corollary \ref{crl:eer}}] To establish an upper bound for the expected excess risk, we consider the following splitting:
\begin{equation}\label{eq:eersplitting}
\E[u( \theta_n^{\lambda})] - \inf_{\theta \in \R^d} u(\theta) = \E[u( \theta_n^{\lambda})] - \E[u( Z_{\infty})] + \E[u( Z_{\infty})] - \inf_{\theta \in \R^d} u(\theta),
\end{equation}
where $Z_{\infty}$ is an $\R^d$-valued random variable with $\mathcal{L}(Z_{\infty}) = \pi_{\beta}$. An upper estimate for the first term on the RHS of \eqref{eq:eersplitting} can be obtained by adapting the proof of \cite[Lemma 4.8]{lim2021nonasymptotic} using Corollary \ref{crl}, Assumptions~\ref{asm:polylip} and \ref{asm:dissip}, which is given by
\[
\E[u( \theta_n^{\lambda})] - \E[u( Z_{\infty})] \leq C_3e^{-C_0\lambda n/2}+C_4\lambda^{1-\epsilon_h/2-\epsilon(1/2-\epsilon_h/4)},
\]
where
\begin{align*}
C_3	&\coloneqq 2^l K_hC_2\rmKL(\pi^{\lambda}_0\|\pi_{\beta})^{1/2}\left( \left(\E[|\theta_0|^{2l+2}]+ \cc_{l+1}(1+1/a)\right)^{1/2}+(2(b+(d+2l)/\beta)/a)^{(l+1)/2}+1\right),\\
C_4	&\coloneqq C_2C_1^{1/2}\left( \left(\E[|\theta_0|^{2l+2}]+ \cc_{l+1}(1+1/a)\right)^{1/2}+(2(b+(d+2l)/\beta)/a)^{(l+1)/2}+1\right)
\end{align*}
with $\cc_{l+1}$ given in \eqref{eq:2pthmmtexpconst}. Moreover, an upper bound for the second term on the RHS of \eqref{eq:eersplitting} can be obtained by adapting the proof of \cite[Lemma 4.9]{lim2021nonasymptotic} using Remark \ref{rmk:polyliph} and Assumption \ref{asm:dissip}, which is given by
\[
\E[u( Z_{\infty})] - \inf_{\theta \in \R^d} u(\theta) \leq C_5/\beta,
\]
where
\[
C_5\equiv C_5(\beta) \coloneqq \frac{d}{2}\log\left(\frac{K_H e (1+4\max\{\sqrt{b/a},\sqrt{2d/(\beta K_H)}\})^l}{a}\left(\frac{\beta b}{d}+1\right)\right)+\log 2.
\]
Combining the two upper bounds yields the desired result.
\end{proof}

\newpage
\appendix
\section{Proof of auxiliary results}\label{appen:aux}
\subsection{Proof of auxiliary results in Section \ref{sec:main}} 
\begin{proof}[\textbf{Proof of statements in Remark \ref{rmk:polyliph}}]\label{rmk:polyliphproof}
For fixed $\theta, \overline{\theta} \in\R^d$, we denote by $g_{\theta, \overline{\theta}}(t) \coloneqq h(t\theta+(1-t)\overline{\theta})$, $t\in[0,1]$, and observe that $g_{\theta, \overline{\theta}}'(t) = H(t\theta+(1-t)\overline{\theta})(\theta-\overline{\theta})$ with $g_{\theta, \overline{\theta}}'$ denoting the derivative of $g_{\theta, \overline{\theta}}$. Then, we obtain that
\[
h(\theta) - h(\overline{\theta}) = g_{\theta, \overline{\theta}}(1) - g_{\theta, \overline{\theta}}(0) = \int_0^1 g_{\theta, \overline{\theta}}'(t)\,\rmd t = \int_0^1 H(t\theta+(1-t)\overline{\theta})(\theta-\overline{\theta})\,\rmd t,
\]
which implies, by using Assumption \ref{asm:polylip}, that
\begin{align*}
|h(\theta) - h(\overline{\theta}) | 
&\leq  \int_0^1 \|H(t\theta+(1-t)\overline{\theta})\|\,\rmd t|\theta-\overline{\theta}|\\
&\leq K_H \int_0^1 (1+|t\theta+(1-t)\overline{\theta}|^{l})\, \rmd t|\theta-\overline{\theta}|\\
&\leq K_H \int_0^1 (1+|t\theta+(1-t)\overline{\theta}|)^{l}\, \rmd t|\theta-\overline{\theta}|\\
&\leq  K_H (1+|\theta|+|\overline{\theta}|)^{l }|\theta-\overline{\theta}|,
\end{align*}
which completes the proof.
\end{proof}

\subsection{Proof of auxiliary results in Section \ref{sec:app}} \label{sec:appapdx}

\begin{proof}[\textbf{Proof of Proposition \ref{prop:dw}}]\label{prop:dwproof}
We note that by Remark \ref{rmk:dwoptconst}, Assumption \ref{asm:polylip} is satisfied with $L=3, l=2, K_H=3, K_h=2$. 

Next, we have, for any $\theta\in\R^d$, that 
\[
\langle h(\theta),\theta\rangle =  \langle \theta(|\theta|^2-1),\theta\rangle \geq |\theta|^2/2-9/4,
\]
which implies Assumption \ref{asm:dissip} holds true with $a=1/2$ and $b=9/4$. Indeed, if $|\theta|\geq \sqrt{6}/2$, we have that
\[
\langle h(\theta),\theta\rangle =|\theta|^2 (  |\theta|^2-1)  \geq |\theta|^2/2 \geq |\theta|^2/2-9/4,
\]
while if $|\theta| \leq \sqrt{6}/2$, we obtain that
\[
\langle h(\theta),\theta\rangle =|\theta|^2 (  |\theta|^2-1)  \geq -|\theta|^2 = |\theta|^2/2-3|\theta|^2/2 \geq |\theta|^2/2-9/4.
\]

Finally, to show Assumption \ref{asm:LSI} holds, we first prove that $u$ defined in \eqref{eq:dwpotential} satisfies the convexity at infinity condition \cite[Eq. (7)]{lytras2025taming}. Indeed, we have, for any $\theta, \overline{\theta}\in\R^d$, that
\begin{align*}
\langle h(\theta)-h(\overline{\theta}),\theta-\overline{\theta}\rangle 
&=  \langle (\theta|\theta|^2-\theta)-(\overline{\theta}|\overline{\theta}|^2-\overline{\theta}),\theta-\overline{\theta}\rangle\\
&= \langle (\theta|\theta|^2-\overline{\theta}|\theta|^2+\overline{\theta}|\theta|^2-\overline{\theta}|\overline{\theta}|^2,\theta-\overline{\theta}\rangle/2\\
&\quad +\langle (\theta|\theta|^2-\theta|\overline{\theta}|^2+\theta|\overline{\theta}|^2-\overline{\theta}|\overline{\theta}|^2,\theta-\overline{\theta}\rangle/2-|\theta-\overline{\theta}|^2\\
&=(|\theta|^2+|\overline{\theta}|^2)|\theta-\overline{\theta}|^2/2+(|\theta|^2-|\overline{\theta}|^2)^2/2-|\theta-\overline{\theta}|^2\\
&\geq ((|\theta|^2+|\overline{\theta}|^2)/2-1)|\theta-\overline{\theta}|^2,
\end{align*}
which implies \cite[Eq. (7)]{lytras2025taming} holds with $c_1=1/2$, $c_2=0$, $c_3=1$, $2r=2$, $0<l<1$. Hence, $\pi_{\beta}$ with $u$ given in \eqref{eq:dwpotential} satisfies our Assumption \ref{asm:LSI} according to \cite[Theorem 5.3]{lytras2025taming}.
\end{proof}

\begin{proof}[\textbf{Proof of Proposition \ref{prop:opt}}]\label{prop:optproof}
First, we show Assumption \ref{asm:polylip} is satisfied. Recall the expression of $u$ given in \eqref{eq:optobj}. For each $i = 1,\dots, d_1$, denote by
\[
\theta_Z^i \coloneqq \langle c_0^{i\cdot}, z\rangle+b_0^i, \quad \sigma_s^i \coloneqq  1/(1+e^{-(\langle c_0^{i\cdot}, z\rangle+b_0^i)}),
\] 
then we have, for any $\theta=(W_1, b_0)\in\R^d$, that
\[
\E\left[(Y-\mathfrak{N}(\theta,Z))^2\right]=\E\left[\left(Y-\sum_{i=1}^{d_1}W_1^{1i}\theta_Z^i \sigma_s^i\right)^2\right].
\]
By using this notation, we have, for any $\theta\in\R^d$, $i,j = 1,\dots, d_1$ with $i\neq j$, that
\begin{align*}
\partial_{W_1^{1i}}u(\theta) &=-2\E\left[(Y-\mathfrak{N}(\theta,Z))\theta_Z^i  \sigma_s^i \right]+\eta W_1^{1i}|\theta|^4,\\
\partial_{b_0^i}u(\theta) &=-2\E\left[(Y-\mathfrak{N}(\theta,Z))W_1^{1i}( \sigma_s^i+\theta_Z^i \sigma_s^i(1- \sigma_s^i))\right]+\eta b_0^i|\theta|^4,\\
\partial_{(W_1^{1i})^2}u(\theta) &=2\E\left[(\theta_Z^i)^2 (\sigma_s^i)^2\right]+\eta |\theta|^4 +4\eta |\theta|^2(W_1^{1i})^2,\\
\partial_{W_1^{1i}W_1^{1j}}u(\theta) &=2\E\left[\theta_Z^i \sigma_s^i\theta_Z^j \sigma_s^j\right]+4\eta |\theta|^2W_1^{1i}W_1^{1j},\\
\partial_{(b_0^i)^2}u(\theta) &=2\E\left[(W_1^{1i})^2( \sigma_s^i+\theta_Z^i \sigma_s^i(1- \sigma_s^i))^2\right] +\eta |\theta|^4 +4\eta |\theta|^2(b_0^i)^2\\
&\quad -2\E\left[(Y-\mathfrak{N}(\theta,Z))W_1^{1i}( 2\sigma_s^i(1- \sigma_s^i)+\theta_Z^i \sigma_s^i(1- \sigma_s^i)(1- 2\sigma_s^i))\right],\\
\partial_{b_0^ib_0^j}u(\theta) &=2\E\left[W_1^{1i}W_1^{1j}(\sigma_s^i+\theta_Z^i \sigma_s^i(1- \sigma_s^i))(\sigma_s^j+\theta_Z^j \sigma_s^j(1- \sigma_s^j))\right]  +4\eta |\theta|^2b_0^ib_0^j,\\
\partial_{W_1^{1i}b_0^i}u(\theta) &=2\E\left[\theta_Z^i \sigma_s^i W_1^{1i}( \sigma_s^i+\theta_Z^i \sigma_s^i(1- \sigma_s^i))\right] \\
&\quad - 2\E\left[(Y-\mathfrak{N}(\theta,Z))( \sigma_s^i+\theta_Z^i \sigma_s^i(1- \sigma_s^i))\right]+4\eta |\theta|^2W_1^{1i}b_0^i,\\
\partial_{W_1^{1i}b_0^j}u(\theta) &=2\E\left[\theta_Z^i \sigma_s^i W_1^{1j}( \sigma_s^j+\theta_Z^j \sigma_s^j(1- \sigma_s^j))\right]+4\eta |\theta|^2W_1^{1i}b_0^j,\\
\partial_{b_0^iW_1^{1j}}u(\theta) &=2\E\left[\theta_Z^j \sigma_s^j W_1^{1i}( \sigma_s^i+\theta_Z^i\sigma_s^i(1- \sigma_s^i))\right]+4\eta |\theta|^2W_1^{1j}b_0^i.
\end{align*}
Then, straight forward calculations yields, for any $\theta, \overline{\theta}\in\R^d$, that
\begin{align*}
\|H(\theta) - H(\overline{\theta})\|
&\leq (320d_1^2+20\eta)\E[1+|X|](1+|c_0|)(1+|\theta|+| \overline{\theta}|)^3|\theta- \overline{\theta}|,\\
\|H(\theta) \| &\leq (256d_1^2+5\eta)\E[(1+|X|)^2](1+|c_0|)^2(1+|\theta|^4),\\
|h(\theta)|&\leq (48d_1+\eta)\E[(1+|X|)^2](1+|c_0|)^2(1+|\theta|^5),
\end{align*}
which implies that Assumption \ref{asm:polylip} holds with $L =(320d_1^2+20\eta)\E[1+|X|](1+|c_0|)$, $l =4$, $K_H =  (256d_1^2+5\eta)\E[(1+|X|)^2](1+|c_0|)^2$, and $K_h=(48d_1+\eta)\E[(1+|X|)^2](1+|c_0|)^2$.

Next, we prove Assumption \ref{asm:dissip} holds. We have, for any $\theta \in\R^d$, that
\begin{align*}
\langle h(\theta),\theta\rangle  
&= \sum_{i=1}^{d_1}\left(W_1^{1i}\partial_{W_1^{1i}}u(\theta)+b_0^i\partial_{b_0^i}u(\theta)\right)\\
&\geq \eta|\theta|^6 - 48d_1\E[(1+|X|)^2](1+|c_0|)^2(1+|\theta|^4)\\
&\geq (\eta/2)|\theta|^2- 96\eta d_1\E[(1+|X|)^2](1+|c_0|)^2/\min\{1,\eta\}\\
&\quad   -96^3d_1^3\E[(1+|X|)^6](1+|c_0|)^6/\min\{1,\eta\}^2,
\end{align*}
where the last inequality hold due to the fact that, if $|\theta|> (96d_1\E[(1+|X|)^2](1+|c_0|)^2/\min\{1,\eta\})^{1/2}$,
\begin{align*}
&\eta|\theta|^6 - 48d_1\E[(1+|X|)^2](1+|c_0|)^2(1+|\theta|^4)\\
&= (\eta/2)|\theta|^6- 48d_1\E[(1+|X|)^2](1+|c_0|)^2 + (\eta/2)|\theta|^6- 48d_1\E[(1+|X|)^2](1+|c_0|)^2 |\theta|^4\\
&\geq  (\eta/2)|\theta|^2- 48d_1\E[(1+|X|)^2](1+|c_0|)^2 \\
&\geq  (\eta/2)|\theta|^2- 96\eta d_1\E[(1+|X|)^2](1+|c_0|)^2/\min\{1,\eta\}\\
&\quad   -96^3d_1^3\E[(1+|X|)^6](1+|c_0|)^6/\min\{1,\eta\}^2,
\end{align*}
while if $|\theta|\leq (96d_1\E[(1+|X|)^2](1+|c_0|)^2/\min\{1,\eta\})^{1/2}$,
\begin{align*}
&\eta|\theta|^6 - 48d_1\E[(1+|X|)^2](1+|c_0|)^2(1+|\theta|^4)\\
&\geq (\eta/2)|\theta|^2 - (\eta/2)|\theta|^2 - 48d_1\E[(1+|X|)^2](1+|c_0|)^2(1+|\theta|^4)\\
&\geq (\eta/2)|\theta|^2 - 96\eta d_1\E[(1+|X|)^2](1+|c_0|)^2/\min\{1,\eta\}\\
&\quad   -96^3d_1^3\E[(1+|X|)^6](1+|c_0|)^6/\min\{1,\eta\}^2.
\end{align*}
Hence, Assumption \ref{asm:dissip} holds with $a=\eta/2$ and $b=96\eta d_1\E[(1+|X|)^2](1+|c_0|)^2/\min\{1,\eta\} +96^3d_1^3\E[(1+|X|)^6](1+|c_0|)^6/\min\{1,\eta\}^2$.

Finally, we show that $u$ defined in \eqref{eq:optobj} satisfies the convexity at infinity condition \cite[Eq. (7)]{lytras2025taming}, which implies that Assumption~\ref{asm:LSI} holds by \cite[Theorem 5.3]{lytras2025taming}. Indeed, for any $\theta, \overline{\theta}=(\overline{W}_1,\overline{b}_0)\in\R^d$, we have that
\begin{align*}
&\langle h(\theta)-h(\overline{\theta}), \theta-\overline{\theta} \rangle\\
&=\sum_{i=1}^{d_1}\left( (\partial_{W_1^{1i}}u(\theta) - \partial_{W_1^{1i}}u(\overline{\theta}))(W_1^{1i} - \overline{W}_1^{1i}) + (\partial_{b_0^i}u(\theta) - \partial_{b_0^i}u(\overline{\theta}))(b_0^i - \overline{b}_0^i)\right)\\
&\geq (\eta/2)(|\theta|^4+|\overline{\theta}|^4)|\theta-\overline{\theta}|^2 - 126d_1\E[(1+|X|)^2](1+|c_0|)^2(1+|\theta|^2+|\overline{\theta}|^2)|\theta-\overline{\theta}|^2,
\end{align*}
which indicates that \cite[Eq. (7)]{lytras2025taming} holds with $c_1=\eta/2$, $c_2=c_3=126d_1\E[(1+|X|)^2](1+|c_0|)^2$, $2r=4$, $l=2$.
\end{proof}

\subsection{Proof of auxiliary results in Section \ref{sec:mtproofs}} \label{sec:mtproofsapdx}
\begin{proof}[\textbf{Proof of Lemma \ref{lem:hlambdaestimates}}]\label{lem:hlambdaestimatesproof}
Recall the definition of $h_{\lambda}$ in \eqref{eq:ktulah}, and let $0<\lambda <1$ and $\epsilon_h \in (0,1/2]$. 

To prove \ref{lem:hlambdaestimates1}, we use Assumption \ref{asm:dissip} to obtain, for any $\theta\in\R^d$, that
\begin{align*}
\langle h_{\lambda}(\theta), \theta\rangle 	
&= \left\langle a\theta +\frac{h(\theta)-a\theta}{(1+\lambda|\theta|^{ (l+1)/\epsilon_h})^{\epsilon_h}}, \theta \right\rangle \\
&= a|\theta|^2+\frac{\langle h(\theta), \theta \rangle -a|\theta|^2}{(1+\lambda|\theta|^{ (l+1)/\epsilon_h})^{\epsilon_h}}\\
&\geq a|\theta|^2+\frac{a|\theta|^2-b -a|\theta|^2}{(1+\lambda|\theta|^{ (l+1)/\epsilon_h})^{\epsilon_h}}\\
&\geq a|\theta|^2-b.
\end{align*}

To prove \ref{lem:hlambdaestimates2}, we use Assumption \ref{asm:polylip} to obtain, for any $\theta\in\R^d$, that
\begin{align*}
|h_{\lambda}(\theta)|
&=\left|a\theta +\frac{h(\theta)-a\theta}{(1+\lambda|\theta|^{ (l+1)/\epsilon_h})^{\epsilon_h}}\right|\\
&\leq 2a|\theta|+\frac{|h(\theta)|}{(1+\lambda|\theta|^{ (l+1)/\epsilon_h})^{\epsilon_h}}\\
&\leq 2a|\theta|+\frac{2^{1-\epsilon_h}K_h(1+|\theta|^{l+1})}{1+\lambda^{\epsilon_h}|\theta|^{l+1}}\\
&\leq 2a|\theta|+2^{1-\epsilon_h} K_h\lambda^{-\epsilon_h}\\
& \leq 2a|\theta|+2K_h\lambda^{-1/2},
\end{align*}
where the second inequality holds due to $(\rho+\iota)^{\epsilon_h} \geq 2^{\epsilon_h-1}(\rho^{\epsilon_h}+\iota^{\epsilon_h})$, $\rho,\iota\geq 0$, for $\epsilon_h \in (0,1/2]$, and the third inequality holds due to $0<\lambda <1$. Moreover, by Assumption \ref{asm:polylip}, we can deduce, for any $\theta\in\R^d$, that
\begin{align*}
|h_{\lambda}(\theta)| \leq 2a|\theta|+|h(\theta)| \leq 2a|\theta|+K_h(1+|\theta|^{l+1})\leq (2a+K_h)(1+|\theta|^{l+1}),
\end{align*}
where the last inequality holds due to $|\theta|\leq 1+|\theta|^{l+1}$ with $l\in\N$.

To prove the first inequality in \ref{lem:hlambdaestimates3}, we fix $\theta, \overline{\theta} \in \R^d$ and denote by $g_{\theta,\overline{\theta},\lambda}(t) \coloneqq h_{\lambda}(t\theta+(1-t)\overline{\theta})$, $t\in [0,1]$. Denoting by $g_{\theta,\overline{\theta},\lambda}'$ its derivative, we have $g_{\theta,\overline{\theta},\lambda}'(t) = \nabla h_{\lambda}(t\theta+(1-t)\overline{\theta})(\theta-\overline{\theta})$, where for all $\theta\in\R^d$,
\begin{equation}\label{eq:nablahlambda}
\nabla h_{\lambda}(\theta) = a\cI_d +\frac{(1+\lambda|\theta|^{ (l+1)/\epsilon_h})(H(\theta)-a\cI_d) -\lambda(l+1) |\theta|^{(l+1)/\epsilon_h-2}(h(\theta)\theta^\cT-a\theta\theta^\cT)}{(1+\lambda|\theta|^{ (l+1)/\epsilon_h})^{1+\epsilon_h}}.
\end{equation}
Moreover, by Assumption \ref{asm:polylip}, it holds, for all $\theta\in\R^d$, that
\begin{align}\label{eq:nablahlbdub}
\|\nabla h_{\lambda}(\theta)\| 
&= \max_{v\in\R^d:|v|=1}|\nabla h_{\lambda}(\theta) v|\nonumber\\
& = \max_{v\in\R^d:|v|=1}\left|av +\frac{(1+\lambda|\theta|^{ (l+1)/\epsilon_h})(H(\theta)v-av) -\lambda(l+1)  |\theta|^{(l+1)/\epsilon_h-2}(h(\theta)-a\theta)\langle \theta,v\rangle)}{(1+\lambda|\theta|^{ (l+1)/\epsilon_h})^{1+\epsilon_h}}\right|\nonumber\\
&\leq 2a+\frac{\|H(\theta)\|}{(1+\lambda|\theta|^{ (l+1)/\epsilon_h})^{ \epsilon_h}}+\frac{\lambda(l+1)|\theta|^{(l+1)/\epsilon_h-1}(|h(\theta)|+a|\theta|)}{(1+\lambda|\theta|^{ (l+1)/\epsilon_h})^{1+\epsilon_h}}\nonumber\\
&\leq 2a+\frac{2^{1-\epsilon_h}\lambda^{-\epsilon_h}K_H(1+|\theta|^{l})}{1+|\theta|^{l+1}}+\frac{\lambda^{-\epsilon_h}(l+1)|\theta|^{(l+1)/\epsilon_h-1}(K_h(1+|\theta|^{l+1})+a|\theta|)}{1+|\theta|^{ (l+1)(1/\epsilon_h+1)}}\nonumber\\
&\leq 2a+4K_H\lambda^{-\epsilon_h}+(l+1)(2K_h+a)\lambda^{-\epsilon_h}\nonumber\\
&\leq \cL_0\lambda^{-\epsilon_h},
\end{align}
where $\cL_0\coloneqq 2a+4K_H+(l+1)(2K_h+a)$, and where the second inequality holds due to $0<\lambda <1$, $\epsilon_h \in (0,1/2]$, $(\rho+\iota)^{\epsilon_h} \geq 2^{\epsilon_h-1}(\rho^{\epsilon_h}+\iota^{\epsilon_h})$, $(\rho+\iota)^{1+\epsilon_h} \geq \rho^{1+\epsilon_h}+\iota^{1+\epsilon_h}$, $\rho,\iota\geq 0$, and the third inequality holds due to $|\theta|^{\rho}\leq 1+|\theta|^{\iota}$, $\iota\geq \rho\geq 0$. 
Thus, we obtain
\begin{equation}\label{eq:hlbdftoc}
|h_{\lambda}(\theta)  - h_{\lambda}(\overline{\theta}) |= \left|\int_0^1  \nabla h_{\lambda}(t\theta+(1-t)\overline{\theta}) \,\rmd t (\theta-\overline{\theta})\right|\leq \cL_0\lambda^{-\epsilon_h}|\theta-\overline{\theta}|.
\end{equation}
Next, we proceed to establish the second inequality in \ref{lem:hlambdaestimates3}. Recall the expression of $\nabla h_{\lambda}$ in \eqref{eq:nablahlambda}, we have, for any $\theta,\overline{\theta}\in\R^d$, that
\begin{align}\label{eq:nablahlambdalipub}
&\|\nabla h_{\lambda}(\theta) - \nabla h_{\lambda}(\overline{\theta})\|\nonumber\\
& = \max_{v\in\R^d:|v|=1}|(\nabla h_{\lambda}(\theta) - \nabla h_{\lambda}(\overline{\theta})) v|\nonumber\\
& =  \max_{v\in\R^d:|v|=1}\left|\left(\frac{H(\theta)-a\cI_d}{(1+\lambda|\theta|^{ (l+1)/\epsilon_h})^{\epsilon_h}} -\frac{H(\overline{\theta})-a\cI_d}{(1+\lambda|\overline{\theta}|^{(l+1)/\epsilon_h})^{\epsilon_h}} \right) v \right. \nonumber\\
&\qquad \left.-\left( \frac{\lambda(l+1) |\theta|^{(l+1)/\epsilon_h-2}(h(\theta)\theta^\cT-a\theta\theta^\cT)}{(1+\lambda|\theta|^{ (l+1)/\epsilon_h})^{1+\epsilon_h}} -  \frac{\lambda(l+1) |\overline{\theta}|^{(l+1)/\epsilon_h-2}(h(\overline{\theta})\overline{\theta}^\cT-a\overline{\theta}\overline{\theta}^\cT)}{(1+\lambda|\overline{\theta}|^{ (l+1)/\epsilon_h})^{1+\epsilon_h}}\right)v\right| \nonumber\\
&\leq \sum_{i=1}^3T_i(\theta,\overline{\theta}),
\end{align}
where
\begin{align*}
T_1(\theta,\overline{\theta})	&\coloneqq  \max_{v\in\R^d:|v|=1}\left| \frac{ H(\theta)v-av }{(1+\lambda|\theta|^{(l+1)/\epsilon_h})^{\epsilon_h}} -\frac{ H(\overline{\theta})v-av }{(1+\lambda|\overline{\theta}|^{(l+1)/\epsilon_h})^{\epsilon_h}}   \right|,\\
T_2(\theta,\overline{\theta}) 	&\coloneqq \max_{v\in\R^d:|v|=1}\left| \frac{\lambda(l+1) |\theta|^{(l+1)/\epsilon_h-2}h(\theta)\langle \theta,v\rangle}{(1+\lambda|\theta|^{(l+1)/\epsilon_h})^{1+\epsilon_h}} -  \frac{\lambda(l+1) |\overline{\theta}|^{(l+1)/\epsilon_h-2}h(\overline{\theta})\langle \overline{\theta},v\rangle}{(1+\lambda|\overline{\theta}|^{(l+1)/\epsilon_h})^{1+\epsilon_h}}   \right|,\\
T_3(\theta,\overline{\theta}) 	&\coloneqq \max_{v\in\R^d:|v|=1}\left| \frac{\lambda(l+1) |\theta|^{(l+1)/\epsilon_h-2}a\theta\langle \theta,v\rangle}{(1+\lambda|\theta|^{(l+1)/\epsilon_h})^{1+\epsilon_h}} -  \frac{\lambda(l+1) |\overline{\theta}|^{(l+1)/\epsilon_h-2}a\overline{\theta}\langle \overline{\theta},v\rangle}{(1+\lambda|\overline{\theta}|^{(l+1)/\epsilon_h})^{1+\epsilon_h}}   \right|.
\end{align*}
By Assumption \ref{asm:polylip}, we obtain, for any $\theta,\overline{\theta}\in\R^d$, that
\begin{align}\label{eq:nablahlambdalipubt1}
T_1(\theta,\overline{\theta})	&= \max_{v\in\R^d:|v|=1}\left| \frac{ (H(\theta)v-av)(1+\lambda|\overline{\theta}|^{(l+1)/\epsilon_h})^{\epsilon_h} - (H(\overline{\theta})v-av)(1+\lambda|\theta|^{(l+1)/\epsilon_h})^{\epsilon_h}}{(1+\lambda|\theta|^{(l+1)/\epsilon_h})^{\epsilon_h}(1+\lambda|\overline{\theta}|^{(l+1)/\epsilon_h})^{\epsilon_h}}    \right| \nonumber\\
&\leq \max_{v\in\R^d:|v|=1}\left| \frac{ (H(\theta)v-av)((1+\lambda|\overline{\theta}|^{(l+1)/\epsilon_h})^{\epsilon_h} - (1+\lambda|\theta|^{(l+1)/\epsilon_h})^{\epsilon_h})}{(1+\lambda|\theta|^{(l+1)/\epsilon_h})^{\epsilon_h}(1+\lambda|\overline{\theta}|^{(l+1)/\epsilon_h})^{\epsilon_h}}    \right| \nonumber\\
&\quad + \max_{v\in\R^d:|v|=1}\left| \frac{ (H(\theta)v - H(\overline{\theta})v )(1+\lambda|\theta|^{(l+1)/\epsilon_h})^{\epsilon_h}}{(1+\lambda|\theta|^{(l+1)/\epsilon_h})^{\epsilon_h}(1+\lambda|\overline{\theta}|^{(l+1)/\epsilon_h})^{\epsilon_h}}    \right| \nonumber\\
&\leq   2\sqrt{2}(2K_H+a)(l+1)(1+|\theta|+|\overline{\theta}|)^{l}|\theta-\overline{\theta}|  + L(1+|\theta|+|\overline{\theta}|)^{l-1}|\theta-\overline{\theta}| \nonumber\\
&\leq (2\sqrt{2}(2K_H+a)(l+1)+L)(1+|\theta|+|\overline{\theta}|)^{l}|\theta-\overline{\theta}|,
\end{align}
where the second inequality holds due to $ (\|H(\theta)\|+a)(1+\lambda|\theta|^{ (l+1)/\epsilon_h})^{-\epsilon_h}\leq 2(2K_H+a)\lambda^{-\epsilon_h}$, $\theta\in\R^d$, and $(1+\lambda|\overline{\theta}|^{(l+1)/\epsilon_h})^{\epsilon_h} - (1+\lambda|\theta|^{(l+1)/\epsilon_h})^{\epsilon_h}  \leq \sqrt{2}\lambda^{\epsilon_h}(l+1) (1+|\theta|+|\overline{\theta}|)^{l}|\theta-\overline{\theta}|$, $ l\in \N$. Moreover, by Remark \ref{rmk:polyliph}, we have, for any $\theta,\overline{\theta}\in\R^d$, that
\begin{align}\label{eq:nablahlambdalipubt2}
T_2(\theta,\overline{\theta})
&\leq \lambda(l+1) \max_{v\in\R^d:|v|=1}\left\{\left| \frac{ |\theta|^{(l+1)/\epsilon_h-2}h(\theta)\langle \theta,v\rangle \left( (1+\lambda|\overline{\theta}|^{(l+1)/\epsilon_h})^{1+\epsilon_h} -  (1+\lambda|\theta|^{(l+1)/\epsilon_h})^{1+\epsilon_h}\right)}{(1+\lambda|\theta|^{(l+1)/\epsilon_h})^{1+\epsilon_h}(1+\lambda|\overline{\theta}|^{(l+1)/\epsilon_h})^{1+\epsilon_h}}    \right| \right. \nonumber\\
&\qquad \left. +\left| \frac{ \left(|\theta|^{(l+1)/\epsilon_h-2} -  |\overline{\theta}|^{(l+1)/\epsilon_h-2}\right)h(\theta)\langle \theta,v\rangle}{(1+\lambda|\overline{\theta}|^{(l+1)/\epsilon_h})^{1+\epsilon_h}}    \right|  
+\left| \frac{|\overline{\theta}|^{(l+1)/\epsilon_h-2}(h(\theta)-h(\overline{\theta}))\langle \theta,v\rangle}{(1+\lambda|\overline{\theta}|^{(l+1)/\epsilon_h})^{1+\epsilon_h}}    \right|  \right.   \nonumber\\
&\qquad \left. +\left| \frac{|\overline{\theta}|^{(l+1)/\epsilon_h-2}h(\overline{\theta})(\langle \theta,v\rangle -  \langle \overline{\theta},v\rangle)}{(1+\lambda|\overline{\theta}|^{(l+1)/\epsilon_h})^{1+\epsilon_h}}    \right|\right\} \nonumber\\
&\leq  \lambda(l+1) \left(\frac{ |\theta|^{(l+1)/\epsilon_h-1}|h(\theta)|(1+\lambda|\overline{\theta}|^{(l+1)/\epsilon_h}) \left| (1+\lambda|\overline{\theta}|^{(l+1)/\epsilon_h})^{\epsilon_h} -  (1+\lambda|\theta|^{(l+1)/\epsilon_h})^{\epsilon_h}\right|}{(1+\lambda|\theta|^{(l+1)/\epsilon_h})^{1+\epsilon_h}(1+\lambda|\overline{\theta}|^{(l+1)/\epsilon_h})^{1+\epsilon_h}}\right. \nonumber\\
&\qquad \left. +\frac{ |\theta|^{(l+1)/\epsilon_h-1}|h(\theta)|  (1+\lambda|\theta|^{(l+1)/\epsilon_h})^{\epsilon_h}\lambda \left| |\overline{\theta}|^{(l+1)/\epsilon_h} -  |\theta|^{(l+1)/\epsilon_h}\right|}{(1+\lambda|\theta|^{(l+1)/\epsilon_h})^{1+\epsilon_h}(1+\lambda|\overline{\theta}|^{(l+1)/\epsilon_h})^{1+\epsilon_h}} \right) \nonumber\\
&\quad +\epsilon_h^{-1}(l+1)^2\max\{K_h, K_H\}  (1+|\theta|+|\overline{\theta}|)^{(l+1)(1/\epsilon_h+1)-2}|\theta-\overline{\theta}|\nonumber\\
&\leq 2(\sqrt{2}+\epsilon_h^{-1})(l+1)^2\max\{K_h, K_H\}  (1+|\theta|+|\overline{\theta}|)^{(l+1)(1/\epsilon_h+1)-2}|\theta-\overline{\theta}|,
\end{align}
where the second and the last inequalities hold due to $\left||\overline{\theta}|^{2\rho} -|\theta|^{2 \rho}\right|  \leq 2\rho (1+|\theta|+|\overline{\theta}|)^{2\rho-1}|\theta-\overline{\theta}|$, $ \rho\geq 1$ and the fact that $|\theta|^{l+1}|h(\theta)|\leq 2K_h(1+|\theta|^{2(l+1)})$, $\theta\in\R^d$. Similarly, we obtain, for any $\theta,\overline{\theta}\in\R^d$, that
\begin{align}\label{eq:nablahlambdalipubt3}
T_3(\theta,\overline{\theta}) 
&\leq \lambda(l+1) \max_{v\in\R^d:|v|=1}\left\{\left| \frac{|\theta|^{(l+1)/\epsilon_h-2}a\theta\langle \theta,v\rangle \left((1+\lambda|\overline{\theta}|^{(l+1)/\epsilon_h})^{1+\epsilon_h} - (1+\lambda|\theta|^{(l+1)/\epsilon_h})^{1+\epsilon_h}\right)}{(1+\lambda|\theta|^{(l+1)/\epsilon_h})^{1+\epsilon_h}(1+\lambda|\overline{\theta}|^{(l+1)/\epsilon_h})^{1+\epsilon_h}}  \right|\right.\nonumber\\
&\qquad \left. +\left| \frac{\left(|\theta|^{(l+1)/\epsilon_h-2} - |\overline{\theta}|^{(l+1)/\epsilon_h-2}\right)a\theta\langle \theta,v\rangle}{(1+\lambda|\overline{\theta}|^{(l+1)/\epsilon_h})^{1+\epsilon_h}}  \right| 
+\left| \frac{|\overline{\theta}|^{(l+1)/\epsilon_h-2}a(\theta-\overline{\theta})\langle \theta,v\rangle }{(1+\lambda|\overline{\theta}|^{(l+1)/\epsilon_h})^{1+\epsilon_h}}  \right| \right. \nonumber\\
&\qquad \left. + \left| \frac{  |\overline{\theta}|^{(l+1)/\epsilon_h-2}a\overline{\theta}(\langle \theta,v\rangle - \langle \overline{\theta},v\rangle)}{(1+\lambda|\overline{\theta}|^{(l+1)/\epsilon_h})^{1+\epsilon_h}}  \right|\right\}\nonumber\\
&\leq \lambda(l+1) \left(  \frac{|\theta|^{(l+1)/\epsilon_h}a(1+\lambda|\overline{\theta}|^{(l+1)/\epsilon_h}) \left|(1+\lambda|\overline{\theta}|^{(l+1)/\epsilon_h})^{\epsilon_h} - (1+\lambda|\theta|^{(l+1)/\epsilon_h})^{\epsilon_h}\right|}{(1+\lambda|\theta|^{(l+1)/\epsilon_h})^{1+\epsilon_h}(1+\lambda|\overline{\theta}|^{(l+1)/\epsilon_h})^{1+\epsilon_h}}\right. \nonumber\\
&\qquad \left. + \frac{|\theta|^{(l+1)/\epsilon_h}a(1+\lambda|\theta|^{(l+1)/\epsilon_h})^{\epsilon_h}\lambda\left||\overline{\theta}|^{(l+1)/\epsilon_h} - |\theta|^{(l+1)/\epsilon_h}\right|}{(1+\lambda|\theta|^{(l+1)/\epsilon_h})^{1+\epsilon_h}(1+\lambda|\overline{\theta}|^{(l+1)/\epsilon_h})^{1+\epsilon_h}}\right)\nonumber\\
&\quad + a\epsilon_h^{-1}(l+1)^2 (1+|\theta|+|\overline{\theta}|)^{(l+1)/\epsilon_h-1}|\theta-\overline{\theta}|\nonumber\\
&\leq (\sqrt{2}+2\epsilon_h^{-1})a(l+1)^2 (1+|\theta|+|\overline{\theta}|)^{(l+1)/\epsilon_h-1}|\theta-\overline{\theta}|.
\end{align}
Substituting \eqref{eq:nablahlambdalipubt1}, \eqref{eq:nablahlambdalipubt2}, and \eqref{eq:nablahlambdalipubt3} back into \eqref{eq:nablahlambdalipub} yields
\[
\|\nabla h_{\lambda}(\theta) - \nabla h_{\lambda}(\overline{\theta})\| \leq \cL_{\nabla,\epsilon_h} (1+|\theta|+|\overline{\theta}|)^{(l+1)(1/\epsilon_h+1)-2}|\theta-\overline{\theta}|,
\]
where $\cL_{\nabla,\epsilon_h}\coloneqq (10\sqrt{2}+4\epsilon_h^{-1})(l+1)^2\max\{K_H,L,K_h,a\}$.

To prove \ref{lem:hlambdaestimates4}, we use Assumption \ref{asm:polylip} to obtain, for any $\theta\in\R^d$, that
\begin{align*}
|h(\theta)-h_{\lambda}(\theta)|^2 	& = \left|\frac{(h(\theta)-a\theta)\left((1+\lambda|\theta|^{(l+1)/\epsilon_h})^{\epsilon_h}-1\right)}{(1+\lambda|\theta|^{(l+1)/\epsilon_h})^{\epsilon_h}}\right|^2\\
&\leq \lambda^2 (|h(\theta)|+a|\theta|)^2|\theta|^{2 (l+1)/\epsilon_h}\\
&\leq 4 \lambda^2(K_h+a)^2(1+ |\theta|^{2(l+1)(1+1/\epsilon_h)}).
\end{align*}
This completes the proof.
\end{proof}

\begin{proof}[\textbf{Proof of Lemma \ref{lem:2ndpthmmt}}]\label{lem:2ndpthmmtproof}
For any $0<\lambda\leq \lambda_{\max}\leq 1$ with $\lambda_{\max}$  given in \eqref{eq:stepsizemax}, $t\in (n, n+1]$, $n \in \N_0$, we define
\begin{equation}\label{eq:delxinotation} 
\Delta_{n,t}^{\lambda}
 \coloneqq \overline{\theta}^{\lambda}_n-\lambda (t-n)h_{\lambda}(\overline{\theta}^{\lambda}_n) ,\quad
\Xi_{n,t}^{\lambda}
\coloneqq \sqrt{2\lambda\beta^{-1}} ( B^{\lambda}_t-B^{\lambda}_n).
\end{equation}
Thus, by \eqref{eq:ktulaproc}, we have that
\begin{equation}\label{eq:ktulaprocnexp}
\overline{\theta}^{\lambda}_t=\Delta_{n,t}^{\lambda} +\Xi_{n,t}^{\lambda}.
\end{equation}

To prove \ref{lem:2ndpthmmti}, we use \eqref{eq:ktulaprocnexp} to obtain that
\begin{equation}\label{eq:2ndmmtexp}
\E\left[\left.|\overline{\theta}^{\lambda}_t|^2\right|\overline{\theta}^{\lambda}_n \right]  = |\Delta_{n,t}^{\lambda}|^2+2\lambda\beta^{-1}(t-n)d,
\end{equation}
where the equality holds due to the fact that $\E\left[\left.\langle \Delta_{n,t}^{\lambda}, \Xi_{n,t}^{\lambda}\rangle \right|\overline{\theta}^{\lambda}_n \right] =0$. Then, we proceed to bound the first term on the RHS of \eqref{eq:2ndmmtexp}. By using \eqref{eq:delxinotation}, Lemma \ref{lem:hlambdaestimates}-\ref{lem:hlambdaestimates1} and \ref{lem:hlambdaestimates}-\ref{lem:hlambdaestimates2}, we obtain that
\begin{align}\label{eq:delub}
 |\Delta_{n,t}^{\lambda}|^2 & =  |\overline{\theta}^{\lambda}_n|^2-2\lambda (t-n)\left\langle h_{\lambda}(\overline{\theta}^{\lambda}_n),\overline{\theta}^{\lambda}_n \right\rangle +\lambda^2 (t-n)^2|h_{\lambda}(\overline{\theta}^{\lambda}_n) |^2 \nonumber\\
 &\leq |\overline{\theta}^{\lambda}_n|^2-2a\lambda (t-n)|\overline{\theta}^{\lambda}_n|^2+2\lambda (t-n)b +\lambda^2 (t-n)  (8a^2|\overline{\theta}^{\lambda}_n|^2+8K_h^2\lambda^{-1})\nonumber\\
 &\leq (1-a\lambda (t-n)) |\overline{\theta}^{\lambda}_n|^2+\lambda (t-n)( 2b+8K_h^2),
\end{align}
where the last inequality holds due to $0<\lambda\leq \lambda_{\max}\leq 1/(8a)$ implying
\[
- a\lambda (t-n)|\overline{\theta}^{\lambda}_n|^2+\lambda^2 (t-n)   8a^2|\overline{\theta}^{\lambda}_n|^2 \leq 0.
\]
Substituting \eqref{eq:delub} into \eqref{eq:2ndmmtexp} yields
\[
\E\left[\left.|\overline{\theta}^{\lambda}_t|^2\right|\overline{\theta}^{\lambda}_n \right]  \leq  (1-a\lambda (t-n)) |\overline{\theta}^{\lambda}_n|^2+\lambda (t-n)\cc_0,
\]
where $\cc_0\coloneq 2b+8K_h^2+2\beta^{-1}d$. This further implies, for $t\in (n, n+1]$, $n \in \N_0$, $0<\lambda\leq \lambda_{\max}\leq 1$, that, 
\begin{align*}
\E\left[|\overline{\theta}^{\lambda}_t|^2  \right] 	&\leq (1-a\lambda (t-n))\E\left[|\overline{\theta}^{\lambda}_n|^2\right]+\lambda (t-n)\cc_0\\
&\leq (1-a\lambda (t-n))(1-a\lambda)\E\left[|\overline{\theta}^{\lambda}_{n-1}|^2\right]+(1-a\lambda (t-n))\lambda\cc_0+\lambda (t-n)\cc_0\\
&\leq (1-a\lambda (t-n))(1-a\lambda)^2\E\left[|\overline{\theta}^{\lambda}_{n-2}|^2\right]+ \lambda((1-a\lambda)+1)\cc_0+\cc_0\\
&\leq \dots\\
&\leq (1-a\lambda (t-n))(1-a\lambda)^n\E\left[|\theta_0|^2\right]+\cc_0(1+1/a).
\end{align*}

Next, we prove the results in \ref{lem:2ndpthmmtii}. For any $p\in [2, \infty)\cap {\N}$, $0<\lambda\leq \lambda_{\max}\leq 1$ with $\lambda_{\max}$  given in \eqref{eq:stepsizemax}, $t\in (n, n+1]$, $n \in \N_0$, we use \eqref{eq:ktulaprocnexp} and follow the same arguments as in the proof of \cite[Lemma 4.2-(ii)]{lim2021nonasymptotic} up to \cite[Eq. (134)]{lim2021nonasymptotic} to obtain
\begin{align}
\begin{split}\label{eq:2pthmmtexp}
\E\left[\left.|\overline{\theta}^{\lambda}_t|^{2p}\right|\overline{\theta}^{\lambda}_n \right] 
&\leq |\Delta_{n,t}^{\lambda}|^{2p} +2^{2p-2}p(2p-1)\lambda(t-n)d\beta^{-1}|\Delta_{n,t}^{\lambda}|^{2p-2}   \\
&\quad + 2^{2p-4}(2p(2p-1))^{p+1}(d\beta^{-1}\lambda(t-n))^p.
\end{split}
\end{align}
By using \eqref{eq:delub}, we obtain that
\begin{align}\label{eq:2pthmmtexpubdelta2p}
|\Delta_{n,t}^{\lambda}|^{2p} 
&\leq \left((1-a\lambda (t-n)) |\overline{\theta}^{\lambda}_n|^2+\lambda (t-n)( 2b+8K_h^2)\right)^p\nonumber\\
&\leq (1+a\lambda (t-n)/2)^{p-1}(1-a\lambda (t-n))^p |\overline{\theta}^{\lambda}_n|^{2p}\nonumber\\
&\quad  +(1+2/(a\lambda (t-n)))^{p-1} (\lambda (t-n))^p( 2b+8K_h^2)^p\nonumber\\
&\leq (1-a\lambda (t-n)/2)^{p-1}(1-a\lambda (t-n))  |\overline{\theta}^{\lambda}_n|^{2p}  +\lambda (t-n)(1+2/a)^{p-1}  ( 2b+8K_h^2)^p,
\end{align}
where the last inequality holds due to $(\rho+\iota)^p\leq (1+\varepsilon)^{p-1}\rho^p+(1+\varepsilon^{-1})^{p-1}\iota^p$, $\rho,\iota \geq 0$, $\varepsilon>0$ with $\varepsilon =a \lambda(t-n)/2 $. Moreover, \eqref{eq:2pthmmtexpubdelta2p} implies, for any $p\in [3, \infty)\cap {\N}$, that
\begin{equation}\label{eq:2pthmmtexpubdelta2p-2}
|\Delta_{n,t}^{\lambda}|^{2p-2} \leq (1-a\lambda (t-n)/2)^{p-2}(1-a\lambda (t-n))  |\overline{\theta}^{\lambda}_n|^{2p-2}  +\lambda (t-n)(1+2/a)^{p-2}  ( 2b+8K_h^2)^{p-1}.
\end{equation}
This, together with \eqref{eq:delub}, further implies \eqref{eq:2pthmmtexpubdelta2p-2} holds for $p\in [2, \infty)\cap {\N}$. Next, we substitute \eqref{eq:2pthmmtexpubdelta2p} and \eqref{eq:2pthmmtexpubdelta2p-2} into \eqref{eq:2pthmmtexp} to obtain
\begin{align}
\E\left[\left.|\overline{\theta}^{\lambda}_t|^{2p}\right|\overline{\theta}^{\lambda}_n \right] 
&\leq (1-a\lambda (t-n)/2)^{p-1}(1-a\lambda (t-n))  |\overline{\theta}^{\lambda}_n|^{2p} \nonumber \\
&\quad +\lambda(t-n)2^{2p-2}p(2p-1)d\beta^{-1}(1-a\lambda (t-n)/2)^{p-2}(1-a\lambda (t-n))|\overline{\theta}^{\lambda}_n|^{2p-2}\nonumber \\
&\quad +\lambda (t-n)(1+2/a)^{p-1} (1+ 2b+8K_h^2)^p(1+2^{2p-1}p(2p-1) d\beta^{-1})\nonumber\\
&\quad + \lambda(t-n)2^{2p-4}(2p(2p-1))^{p+1}(d\beta^{-1} )^p\nonumber\\
\begin{split}\label{eq:2pthmmtexpub}
&\leq (1-a\lambda (t-n))  |\overline{\theta}^{\lambda}_n|^{2p}+ \lambda(t-n)\overline{\cc}_p \\
&\quad -(a\lambda (t-n)/4) (1-a\lambda (t-n)/2)^{p-2}(1-a\lambda (t-n))  |\overline{\theta}^{\lambda}_n|^{2p} \\
&\quad +\lambda(t-n)2^{2p-2}p(2p-1)d\beta^{-1}(1-a\lambda (t-n)/2)^{p-2}(1-a\lambda (t-n))|\overline{\theta}^{\lambda}_n|^{2p-2},
\end{split}
\end{align}
where $\overline{\cc}_p \coloneqq (1+2/a)^{p-1} (1+ 2b+8K_h^2)^p(1+2^{2p-1}p(2p-1) d\beta^{-1}) +2^{2p-4}(2p(2p-1))^{p+1}(d\beta^{-1} )^p$. We note that, for $|\theta|>(2^{2p}p(2p-1)d(a\beta)^{-1})^{1/2}\eqqcolon \cM(p)$,
\[
-(a/4)|\theta|^{2p}  +  2^{2p-2}p(2p-1)d\beta^{-1}|\theta|^{2p-2}\leq 0.
\]
Thus, by denoting $\cS_{n,\cM(p)}:=\{\omega \in \Omega:|\overline{\theta}^{\lambda}_n(\omega)|>\cM(p)\}$ and by using \eqref{eq:2pthmmtexpub}, we obtain
\begin{align}\label{eq:2pthmmtexpubS}
\E\left[\left.|\overline{\theta}^{\lambda}_t|^{2p}\1_{\cS_{n,\cM(p)}}\right|\overline{\theta}^{\lambda}_n \right] 
&\leq (1-a\lambda (t-n))  |\overline{\theta}^{\lambda}_n|^{2p}\1_{\cS_{n,\cM(p)}}+ \lambda(t-n)\cc_p \1_{\cS_{n,\cM(p)}},
\end{align}
and similarly
\begin{align}\label{eq:2pthmmtexpubSc}
\E\left[\left.|\overline{\theta}^{\lambda}_t|^{2p}\1_{\cS_{n,\cM(p)}^{\cc}}\right|\overline{\theta}^{\lambda}_n \right] 
&\leq (1-a\lambda (t-n))  |\overline{\theta}^{\lambda}_n|^{2p}\1_{\cS_{n,\cM(p)}^{\cc}}+ \lambda(t-n)\cc_p \1_{\cS_{n,\cM(p)}^{\cc}},
\end{align}
where $\cc_p \coloneqq \overline{\cc}_p+a(\cM(p))^{2p}$. Finally, combining \eqref{eq:2pthmmtexpubS} and \eqref{eq:2pthmmtexpubSc} yields the desired result, i.e., for any $p\in [2, \infty)\cap {\N}$, $0<\lambda\leq \lambda_{\max}\leq 1$, $t\in (n, n+1]$, $n \in \N_0$,
\[
\E\left[\left.|\overline{\theta}^{\lambda}_t|^{2p}\right|\overline{\theta}^{\lambda}_n \right] 
\leq (1-a\lambda (t-n))  |\overline{\theta}^{\lambda}_n|^{2p} + \lambda(t-n)\cc_p ,
\]
which implies
\[
\E\left[|\overline{\theta}^{\lambda}_t|^{2p}  \right] 
\leq (1-a\lambda (t-n))  (1-a\lambda)^n\E\left[|\theta_0|^{2p} \right]+ \cc_p(1+1/a),
\]
where
\begin{align}
\begin{split}\label{eq:2pthmmtexpconst} 
\cc_p&\coloneqq (1+2/a)^{p-1} (1+ 2b+8K_h^2)^p(1+2^{2p-1}p(2p-1) d\beta^{-1}) \\
&\qquad +2^{2p-4}(2p(2p-1))^{p+1}(\max\{1,d\beta^{-1} \})^p+a(\max\{1,2^{2p}p(2p-1)d(a\beta)^{-1}\})^{p}.
\end{split}
\end{align}
This completes the proof.
\end{proof}

\begin{proof}[\textbf{Proof of Lemma \ref{lem:ktuladensitygrowth}}]\label{lem:ktuladensitygrowthproof}
First, we note that the result holds for $n=0$ under Assumption \ref{asm:initialc}.

Then, for any $0<\lambda\leq \lambda_{\max}\leq 1$ with $\lambda_{\max}$  given in \eqref{eq:stepsizemax}, $t\in (n, n+1]$, $n \in \N_0$, by using \eqref{eq:delxinotation}, \eqref{eq:ktulaprocnexp}, the fact that $\Delta_{n,t}^{\lambda}$ and $\Xi_{n,t}^{\lambda}$ are independent, and that $\Xi_{n,t}^{\lambda}$ is normally distributed with mean 0 and covariance matrix $2\lambda(t-n)\beta^{-1}\cI_d$, we apply \cite[Proposition 2]{polyanskiy2016wasserstein} with \[
\sigma^2 \curvearrowleft 2\lambda(t-n)\beta^{-1}, \quad \E[|B|] \curvearrowleft \left(\E\left[|\overline{\theta}^{\lambda}_t|^2\right]\right)^{1/2}
\]
to obtain that, for any $n\in\N_0$,
\[
|\nabla \log \pi^{\lambda}_{n+1}(\theta)| \leq \lambda^{-1}(\widetilde{\cc}_1|\theta|+\widetilde{\cc}_2),
\]
where $\widetilde{\cc}_1 \coloneqq 2\beta $ and $\widetilde{\cc}_2 \coloneqq 2\beta (\E\left[|\theta_0|^2\right] +\cc_0(1+1/a))^{1/2}$ with $\cc_0$ given in Lemma \ref{lem:2ndpthmmt}-\ref{lem:2ndpthmmti}.
\end{proof}

\begin{proof}[\textbf{Proof of Lemma \ref{lem:decomptermsF1}}]\label{lem:decomptermsF1proof}
By using the expression of $F_1$ in \eqref{eq:decompterms}, we have, for any $\theta\in\R^d$, $0<\lambda\leq \lambda_{\max}$, $n\in\N_0$, $t\in(n ,n+1 ]$, that
\begin{align*}
F_1( \theta)		
&=\int_{\R^d} \left(\cI_d-\lambda(t-n ) \nabla h_{\lambda}(y)\right)\left(y-\theta-\lambda(t-n )h_{\lambda}(y)\right) \pi^{\lambda}_{n|t}(y|\theta)\, \rmd y \\
&=\int_{\R^d} \left(\cI_d-\lambda(t-n ) \nabla h_{\lambda}(y)\right)\left(y-\theta-\lambda(t-n )h_{\lambda}(y)\right) \pi^{\lambda}_{t|n}(\theta|y)\frac{\pi^{\lambda}_n(y)}{\pi^{\lambda}_t(\theta)}\, \rmd y  ,
\end{align*}
where $\pi^{\lambda}_{n|t}$ denotes the conditional density of $\overline{\theta}^{\lambda}_n$ given $\overline{\theta}^{\lambda}_t$ while $\pi^{\lambda}_{t|n}$ denotes the conditional density of $\overline{\theta}^{\lambda}_t$ given $\overline{\theta}^{\lambda}_n$, and where the second equality hold due to Bayes' rule. By the fact that 
\[
\pi^{\lambda}_{t|n}(\theta|y) = \phi\left(\frac{\theta -y+\lambda(t-n)h_{\lambda}(y)}{\sqrt{2\beta^{-1}\lambda(t-n)}}\right)
\]
for any $\theta,y\in\R^d$ with $\phi$ denoting the $d$-dimensional standard normal density, the above equality can be rewritten as
\begin{align}\label{eq:F1exp}
F_1( \theta)		
&=\int_{\R^d} \left(\cI_d-\lambda(t-n ) \nabla h_{\lambda}(y)\right)\left(y-\theta-\lambda(t-n )h_{\lambda}(y)\right) \phi\left(\frac{\theta -y+\lambda(t-n)h_{\lambda}(y)}{\sqrt{2\beta^{-1}\lambda(t-n)}}\right)\frac{\pi^{\lambda}_n(y)}{\pi^{\lambda}_t(\theta)}\, \rmd y  \nonumber\\
&= -2\beta^{-1}\lambda(t-n ) \int_{\R^d} \nabla_y \phi\left(\frac{\theta -y+\lambda(t-n)h_{\lambda}(y)}{\sqrt{2\beta^{-1}\lambda(t-n)}}\right)\frac{\pi^{\lambda}_n(y)}{\pi^{\lambda}_t(\theta)}\, \rmd y \nonumber\\
& = 2\beta^{-1}\lambda(t-n ) \int_{\R^d}  \phi\left(\frac{\theta -y+\lambda(t-n)h_{\lambda}(y)}{\sqrt{2\beta^{-1}\lambda(t-n)}}\right)\frac{\nabla \pi^{\lambda}_n(y)}{\pi^{\lambda}_t(\theta)}\, \rmd y\nonumber\\
& =2\beta^{-1} \lambda(t-n ) \int_{\R^d} \pi^{\lambda}_{t|n}(\theta|y)\frac{\pi^{\lambda}_n(y)  \nabla\log \pi^{\lambda}_n(y)}{\pi^{\lambda}_t(\theta)}\, \rmd y \nonumber\\
& =2\beta^{-1}\lambda(t-n )\E\left[\left. \nabla \log \pi^{\lambda}_n(\overline{\theta}^{\lambda}_n) \right|\overline{\theta}^{\lambda}_t=\theta\right],
\end{align}
where the second equality holds due to the chain rule and the fact that $\nabla \phi(y)=-y\phi(y)$, $y\in\R^d$, while the third equality is obtained by applying integration by parts \cite[Appendix E]{mou2022improvedsupp}. By using \eqref{eq:F1exp}, Lemma~\ref{lem:ktuladensitygrowth}, and by applying H\"{o}lder's inequality twice, we obtain, for any $0<\lambda\leq \lambda_{\max}$, $n\in\N_0$, $t\in(n ,n+1 ]$, that
\begin{align}\label{eq:F1ub}
&\E \left[\left| \nabla h_{\lambda}(\overline{\theta}^{\lambda}_t) F_1(\overline{\theta}^{\lambda}_t)\right|^2\right] \nonumber\\
& \leq 4\beta^{-2} \lambda^2\E \left[\left| \nabla h_{\lambda}(\overline{\theta}^{\lambda}_t) \nabla \log \pi^{\lambda}_n(\overline{\theta}^{\lambda}_n)\right|^2\right] \nonumber\\
& \leq 4\beta^{-2}\lambda^2\E \left[\| \nabla h_{\lambda}(\overline{\theta}^{\lambda}_t)\|^2| \nabla \log \pi^{\lambda}_n(\overline{\theta}^{\lambda}_n)|^{2-\epsilon}\lambda^{-\epsilon}(\widetilde{\cc}_1|\overline{\theta}^{\lambda}_n|+\widetilde{\cc}_2)^{\epsilon}\right] \nonumber\\
& \leq 4\beta^{-2} \lambda^{2-\epsilon}\left(\E\left[| \nabla \log \pi^{\lambda}_n(\overline{\theta}^{\lambda}_n)|^2\right]\right)^{1-\epsilon/2}\left(\E\left[\| \nabla h_{\lambda}(\overline{\theta}^{\lambda}_t)\|^{4/\epsilon}(\widetilde{\cc}_1|\overline{\theta}^{\lambda}_n|+\widetilde{\cc}_2)^2\right]\right)^{\epsilon/2}\nonumber\\
&\leq 4\beta^{-2}\lambda^{2-\epsilon} J_n^{1-\epsilon/2}\left(\E\left[\| \nabla h_{\lambda}(\overline{\theta}^{\lambda}_t)\|^{8/\epsilon}\right]\right)^{\epsilon/4}\left(\E\left[(\widetilde{\cc}_1|\overline{\theta}^{\lambda}_n|+\widetilde{\cc}_2)^4\right]\right)^{\epsilon/4},
\end{align}
where $\epsilon>0$. To establish an upper bound for the first expectation on the RHS of \eqref{eq:F1ub}, we use Lemma~\ref{lem:hlambdaestimates}-\ref{lem:hlambdaestimates3} to obtain, for any $\theta\in\R^d$, that
\[
\|\nabla h_{\lambda}(\theta)\|\leq \cK_{\nabla,\epsilon_h}(1+|\theta|)^{(l+1)/\epsilon_h+l},
\]
where $\cK_{\nabla,\epsilon_h}\coloneqq (10\sqrt{2}+4\epsilon_h^{-1})(l+1)^2\max\{K_H,L,K_h,a,\|\nabla h_{\lambda}(0)\|,1\}$. This, together with Lemma~\ref{lem:2ndpthmmt}, implies, for any $t\geq 0$, that
\begin{align}\label{eq:nablahlbdaue}
\begin{split}
\left(\E\left[\| \nabla h_{\lambda}(\overline{\theta}^{\lambda}_t)\|^{8/\epsilon}\right]\right)^{\epsilon/4} 
&\leq 2^{2((l+1)/\epsilon_h+l)+\epsilon/2}\cK_{\nabla,\epsilon_h}^2\cc_{\lceil4((l+1)/\epsilon_h+l)/\epsilon\rceil}^{\epsilon/4}(1+1/a)^{\epsilon/4}\\
&\qquad \times \left(1+\E\left[|\theta_0|^{2\lceil 4((l+1)/\epsilon_h+l)/\epsilon\rceil}\right]\right)^{\epsilon/4}.
\end{split}
\end{align}
Moreover, to upper bound the second expectation on the RHS of \eqref{eq:F1ub}, we apply Lemma \ref{lem:2ndpthmmt} and the expressions of $\widetilde{\cc}_1$ and $\widetilde{\cc}_2$ given in Lemma \ref{lem:ktuladensitygrowth} to obtain that
\[
\left(\E\left[(\widetilde{\cc}_1|\overline{\theta}^{\lambda}_n|+\widetilde{\cc}_2)^4\right]\right)^{\epsilon/4} \leq 2^{9\epsilon/4}\beta^{\epsilon}\cc_2^{\epsilon/4}(1+1/a)^{\epsilon/2}\left(1+\E\left[|\theta_0|^4\right]\right)^{\epsilon/4}.
\]
Hence, we obtain, for any $\epsilon>0$, $0<\lambda\leq \lambda_{\max}$, $n\in\N_0$, $t\in(n ,n+1 ]$, that
\[
\E \left[\left| \nabla h_{\lambda}(\overline{\theta}^{\lambda}_t) F_1(\overline{\theta}^{\lambda}_t)\right|^2\right]
\leq \cC_{D,\epsilon,\epsilon_h}\lambda^{2-\epsilon} J_n^{1-\epsilon/2},
\]
where
\begin{align}\label{eq:cDepsilonexp}
\begin{split}
\cC_{D,\epsilon,\epsilon_h}
&\coloneqq 2^{2(l+1)(1/\epsilon_h+1)+13\epsilon/4}\cK_{\nabla,\epsilon_h}^2\beta^{\epsilon-2}(1+1/a)^{3\epsilon/4}\cc_2^{\epsilon/4}\cc_{\lceil 4((l+1)/\epsilon_h+l)/\epsilon\rceil}^{\epsilon/4}\\
&\qquad \times\left(1+\E\left[|\theta_0|^{\max\{2\lceil 4((l+1)/\epsilon_h+l)/\epsilon\rceil,4\}}\right]\right)^{\epsilon/2}
\end{split}
\end{align}
with $\cK_{\nabla,\epsilon_h}\coloneqq (10\sqrt{2}+4\epsilon_h^{-1})(l+1)^2\max\{K_H,L,K_h,a,\|\nabla h_{\lambda}(0)\|,1\}$ and $\cc_p$, $p\in [2, \infty)\cap {\N}$, given in \eqref{eq:2pthmmtexpconst}. 
We note that $\cC_{D,\epsilon,\epsilon_h}>0$ is a finite constant due to Assumption \ref{asm:initialc}.
\end{proof}

\begin{proof}[\textbf{Proof of Lemma \ref{lem:ubforfisherinfo}}]\label{lem:ubforfisherinfoproof}
For any $0<\lambda\leq \lambda_{\max}$ with $\lambda_{\max}$  given in \eqref{eq:stepsizemax}, define $f\colon \R^d \to \R^d$ by 
\begin{equation}\label{eq:deff}
f(\theta)\coloneqq  \theta-\lambda h_{\lambda}(\theta).
\end{equation}
We note that, by Remark \ref{lem:hlambdaestimates}-\ref{lem:hlambdaestimates3}, $f$ is a bi-Lipschitz mapping, i.e., for any $\theta,\overline{\theta}\in\R^d$,
\begin{equation}\label{eq:fbilip}
(1-\cL_0 \lambda^{1-\epsilon_h})|\theta - \overline{\theta}| \leq \|f(\theta) - f(\overline{\theta})\| \leq (1+\cL_0 \lambda^{1-\epsilon_h})|\theta - \overline{\theta}|
\end{equation}
with $\cL_0 \coloneqq 2a+4K_H+(l+1)(2K_h+a)$. Denote by $\psi_n^{\lambda}$ the density of $f(\overline{\theta}^{\lambda}_n)$, $n \in \N_0$, and $\phi^{\lambda}_{\beta}$ the density of the normal distribution with mean 0 and variance $2\lambda\beta^{-1}\cI_d$. Recall the expression of $(\overline{\theta}^{\lambda}_t)_{t \geq 0 }$ in \eqref{eq:ktulaproc}, we have that $\pi_{n+1}^{\lambda} =\psi_n^{\lambda}* \phi^{\lambda}_{\beta}$. 

Our aim is to provide a uniform upper bound for $J_n$, $n \in \N_0$. To this end, we use the definition of $J_n$ in \eqref{eq:fisherinfo} and \cite[Proposition 3]{rioul2010information} to obtain, for any $0<\lambda\leq \lambda_{\max}$, $n\in\N_0$, that
\begin{equation}\label{eq:fisherinfoconvoineq}
\frac{1}{J_{n+1}} \geq \frac{1}{J(\psi_n^{\lambda})} +\frac{1}{J(\phi^{\lambda}_{\beta})} \geq \frac{1}{J(\psi_n^{\lambda})} +\frac{2\lambda}{d\beta}.
\end{equation}
To establish an upper bound for $J(\psi_n^{\lambda})$, we use transformation of variables, i.e., for any $Z =f(\overline{\theta}^{\lambda}_n)$,
\[
\psi_n^{\lambda}(z) = \frac{\pi_n^{\lambda}(f^{-1}(z))}{\rmdet(\nabla f(f^{-1}(z)))},
\]
to obtain that
\begin{align}\label{eq:fisherinfopsiub}
J(\psi_n^{\lambda}) 
&= \int_{\R^d}|\nabla \log \psi_n^{\lambda}(z)|^2\psi_n^{\lambda}(z)\,\rmd z \nonumber\\
&= \int_{\R^d}|\nabla \log \psi_n^{\lambda}(f(\theta))|^2\psi_n^{\lambda}(f(\theta))\rmdet(\nabla f(\theta)) \,\rmd \theta \nonumber\\
&= \int_{\R^d}|(\nabla f(\theta))^{-1}\nabla_{\theta} \log \psi_n^{\lambda}(f(\theta))|^2\pi_n^{\lambda}(\theta) \,\rmd \theta \nonumber\\
&= \int_{\R^d}\left|(\nabla f(\theta))^{-1}\left(\nabla_{\theta} \log \pi_n^{\lambda}(\theta) - \nabla_{\theta} \log \rmdet(\nabla f(\theta))    \right)\right|^2\pi_n^{\lambda}(\theta) \,\rmd \theta \nonumber\\
\begin{split}
&\leq (1+\cL_0\lambda^{1-\epsilon_h})\int_{\R^d}\left|(\nabla f(\theta))^{-1} \nabla_{\theta} \log \pi_n^{\lambda}(\theta)   \right|^2\pi_n^{\lambda}(\theta) \,\rmd \theta \\
&\quad+ (1+1/(\cL_0\lambda^{1-\epsilon_h}))\int_{\R^d}\left|(\nabla f(\theta))^{-1}  \nabla_{\theta} \log \rmdet(\nabla f(\theta))  \right|^2\pi_n^{\lambda}(\theta) \,\rmd \theta,
\end{split}
\end{align}
the last inequality holds due to $(\rho+\iota)^2\leq (1+\varepsilon) \rho^2+(1+\varepsilon^{-1}) \iota^2$, $\rho,\iota \geq 0$, $\varepsilon>0$ with $\varepsilon =\cL_0\lambda^{1-\epsilon_h} $. Then, by using \eqref{eq:fbilip} and the fact that $0<\lambda\leq \lambda_{\max}\leq 1/(6\cL_0)^{1/(1-\epsilon_h)}$, we have that
\begin{equation}\label{eq:fisherinfopsiub1}
(1+\cL_0\lambda^{1-\epsilon_h})\|(\nabla f(\theta))^{-1}\|^2\leq (1+\cL_0\lambda^{1-\epsilon_h})/(1-\cL_0\lambda^{1-\epsilon_h})^2\leq (1+(6/5)\cL_0\lambda^{1-\epsilon_h})^3\leq 1+5\cL_0\lambda^{1-\epsilon_h}.
\end{equation}
Moreover, by Remark \ref{lem:hlambdaestimates}-\ref{lem:hlambdaestimates3}, we obtain, for any $i,j=1,\dots,d$, $\theta\in\R^d$, that
\begin{align*}
|\langle \nabla(\nabla h_{\lambda})^{(i,j)}(\theta),v\rangle|
& =\left| \lim_{h\to 0} \frac{(\nabla h_{\lambda})^{(i,j)}(\theta+hv) - (\nabla h_{\lambda})^{(i,j)}(\theta)}{h}  \right|\\
&\leq \lim_{h\to 0}\cL_{\nabla,\epsilon_h}(1+|\theta+hv|+|\theta|)^{(l+1)(1/\epsilon_h+1)-2} |v|\\
&\leq 2^{(l+1)(1/\epsilon_h+1)-2}\cL_{\nabla,\epsilon_h}(1+|\theta|)^{(l+1)(1/\epsilon_h+1)-2} |v|,
\end{align*}
which implies
\[
|\nabla(\nabla h_{\lambda})^{(i,j)}(\theta)| \leq 2^{(l+1)(1/\epsilon_h+1)-2}\cL_{\nabla,\epsilon_h}(1+|\theta|)^{(l+1)(1/\epsilon_h+1)-2}.
\]
This, \eqref{eq:fbilip}, the fact that $1/(1-\cL_0\lambda^{1-\epsilon_h})\leq 2$, and Lemma \ref{lem:2ndpthmmt} allow us to further deduce that
\begin{align}\label{eq:fisherinfopsiub2}
&\int_{\R^d}\left|(\nabla f(\theta))^{-1}  \nabla_{\theta} \log \rmdet(\nabla f(\theta))  \right|^2\pi_n^{\lambda}(\theta) \,\rmd \theta \nonumber\\
&=\int_{\R^d}\left|(\nabla f(\theta))^{-2}  \nabla \cdot (\cI_d - \lambda \nabla h_{\lambda}(\theta))   \right|^2\pi_n^{\lambda}(\theta) \,\rmd \theta \nonumber\\
&\leq 2^{2(l+1)(1/\epsilon_h+1)}\lambda^2d^2\cL_{\nabla,\epsilon_h}^2\E\left[(1+|\overline{\theta}^{\lambda}_n|)^{2\lceil (l+1)(1/\epsilon_h+1)-2\rceil}\right]\nonumber\\
&\leq 2^{4(l+1)(1/\epsilon_h+1)-2}\lambda^2d^2\cL_{\nabla,\epsilon_h}^2\cc_{\lceil (l+1)(1/\epsilon_h+1)-2\rceil}(1+1/a)\left(1+\E\left[|\theta_0|^{2\lceil (l+1)(1/\epsilon_h+1)-2\rceil}\right]\right).
\end{align}
Substituting \eqref{eq:fisherinfopsiub1} and \eqref{eq:fisherinfopsiub2} into \eqref{eq:fisherinfopsiub} and noticing $\cL_0\lambda^{1-\epsilon_h}\leq 1$ yield
\[
J(\psi_n^{\lambda}) \leq (1+5\cL_0\lambda^{1-\epsilon_h})J_n+\cC_{\psi,\epsilon_h}\lambda^{1+\epsilon_h}/\cL_0,
\]
where $\cC_{\psi,\epsilon_h} \coloneqq 2^{4(l+1)(1/\epsilon_h+1)-1} d^2\cL_{\nabla,\epsilon_h}^2\cc_{\lceil (l+1)(1/\epsilon_h+1)-2\rceil}(1+1/a)\left(1+\E\left[|\theta_0|^{2\lceil (l+1)(1/\epsilon_h+1)-2\rceil}\right]\right)$. Hence, by plugging the above result back into \eqref{eq:fisherinfoconvoineq}, we obtain, for any $0<\lambda\leq \lambda_{\max}$, $n\in\N_0$, that
\[
\frac{1}{J_{n+1}} \geq \frac{1}{(1+5\cL_0\lambda^{1-\epsilon_h})J_n+\cC_{\psi,\epsilon_h}\lambda^{1+\epsilon_h}/\cL_0} +\frac{2\lambda}{d\beta}.
\]
We note that the above lower bound can be simplified as
\begin{equation}\label{eq:Jnrecursion}
\frac{1}{J_{n+1}} \geq \min\left\{\frac{\cL_0^2}{2\cC_{\psi,\epsilon_h}\lambda^{2\epsilon_h}}, \frac{(1-6\cL_0\lambda^{1-\epsilon_h})}{J_n } +\frac{2\lambda}{d\beta}\right\}.
\end{equation}
Indeed, if $J_n \geq \cC_{\psi,\epsilon_h}\lambda^{2\epsilon_h}/\cL_0^2$, then we have that
\[
\frac{1}{J_{n+1}} \geq \frac{1}{(1+6\cL_0\lambda^{1-\epsilon_h})J_n} +\frac{2\lambda}{d\beta}\geq \frac{1-6\cL_0\lambda^{1-\epsilon_h}}{J_n} +\frac{2\lambda}{d\beta},
\]
and if $J_n \leq \cC_{\psi,\epsilon_h}\lambda^{2\epsilon_h}/\cL_0^2$, then, by $0<\lambda\leq \lambda_{\max}\leq 1/(6\cL_0)^{1/(1-\epsilon_h)}$, it is clear that
\[
\frac{1}{J_{n+1}} \geq \frac{1}{(1+5\cL_0\lambda^{1-\epsilon_h})\cC_{\psi,\epsilon_h}\lambda^{2\epsilon_h}/\cL_0^2+\cC_{\psi,\epsilon_h}\lambda^{2\epsilon_h}/(6\cL_0^2)} \geq \frac{\cL_0^2}{2\cC_{\psi,\epsilon_h}\lambda^{2\epsilon_h}}.
\]
Finally, by applying \cite[Lemma 7]{mou2022improved} to \eqref{eq:Jnrecursion} with
\[
u_k \curvearrowleft 1/J_n, \quad \lambda_1 \curvearrowleft \cL_0^2/(2\cC_{\psi,\epsilon_h}\lambda^{2\epsilon_h}), \quad \lambda_2 \curvearrowleft \lambda^{\epsilon_h}/(3d\beta\cL_0), \quad \gamma \curvearrowleft 6\cL_0\lambda^{1-\epsilon_h},
\]
we obtain, for any $0<\lambda\leq \lambda_{\max}$, $n\in\N_0$, that
\[
1/J_n\geq \min\{1/J_0, \cL_0^2/(4\cC_{\psi,\epsilon_h}\lambda^{2\epsilon_h}),\lambda^{\epsilon_h}/(6d\beta\cL_0)\},
\]
which implies
\[
J_n \leq \max\{J_0, 4\cC_{\psi,\epsilon_h}\lambda^{2\epsilon_h}/\cL_0^2, 6d\beta\cL_0\lambda^{-\epsilon_h} \}\leq \cC_{J,\epsilon_h}	 \lambda^{-\epsilon_h},
\]
where 
\begin{align}\label{eq:cJexp}
\begin{split}
\cC_{J,\epsilon_h}		& \coloneqq J_0+4\cC_{\psi,\epsilon_h}/\cL_0^2+ 6d\beta\cL_0,\\
\cC_{\psi,\epsilon_h}	& \coloneqq 2^{4(l+1)(1/\epsilon_h+1)-1} d^2\cL_{\nabla,\epsilon_h}^2\cc_{\lceil (l+1)(1/\epsilon_h+1)-2\rceil}(1+1/a)\left(1+\E\left[|\theta_0|^{2\lceil (l+1)(1/\epsilon_h+1)-2\rceil}\right]\right),\\
\cL_0				& \coloneqq 2a+4K_H+(l+1)(2K_h+a),\\
\cL_{\nabla,\epsilon_h}	& \coloneqq  (10\sqrt{2}+4\epsilon_h^{-1})(l+1)^2\max\{K_H,L,K_h,a\},
\end{split}
\end{align}
and where $\cc_p$, $p\in [2, \infty)\cap {\N}$, is given in \eqref{eq:2pthmmtexpconst}.
\end{proof}

\begin{proof}[\textbf{Proof of Lemma \ref{lem:decomptermsF2n3}}]\label{lem:decomptermsF2n3proof}
Recall the expression of $F_2$ in \eqref{eq:decompterms} and the definition of $(\overline{\theta}^{\lambda}_t)_{t \geq 0 }$ in \eqref{eq:ktulaproc}. We have, for any $0<\lambda\leq \lambda_{\max}\leq 1$, $n\in\N_0$, $t\in(n ,n+1 ]$, that
\begin{align*}
\E \left[\left| \nabla h_{\lambda}(\overline{\theta}^{\lambda}_t) F_2(\overline{\theta}^{\lambda}_t)\right|^2\right] 
&\leq \E \left[\left| \nabla h_{\lambda}(\overline{\theta}^{\lambda}_t)\lambda(t-n )(-\nabla h_{\lambda}(\overline{\theta}^{\lambda}_n))\left(\overline{\theta}^{\lambda}_n-\overline{\theta}^{\lambda}_t-\lambda(t-n )  h_{\lambda}(\overline{\theta}^{\lambda}_n)\right)\right|^2\right]\\
&\leq \lambda^2\E \left[\| \nabla h_{\lambda}(\overline{\theta}^{\lambda}_t)\|^2\|\nabla h_{\lambda}(\overline{\theta}^{\lambda}_n)\|^2|\sqrt{2\lambda\beta^{-1}} ( B^{\lambda}_t- B^{\lambda}_n)|^2\right]\\
&\leq 2 \beta^{-1} \cL_0^2\lambda^2\left(\E \left[\| \nabla h_{\lambda}(\overline{\theta}^{\lambda}_t)\|^4\right]\right)^{1/2} \left(\E \left[|  B^{\lambda}_t- B^{\lambda}_n|^4\right]\right)^{1/2},
\end{align*}
where the last inequality is obtained by applying \eqref{eq:nablahlbdub} and Cauchy-Schwarz inequality. Thus, for any $0<\lambda\leq \lambda_{\max}\leq 1$, $n\in\N_0$, $t\in(n ,n+1 ]$, by applying \eqref{eq:nablahlbdaue} with $\epsilon=2$ and by noticing $\E \left[|  B^{\lambda}_t- B^{\lambda}_n|^4\right]\leq d(2+d)$, we obtain that
\[
\E \left[\left| \nabla h_{\lambda}(\overline{\theta}^{\lambda}_t) F_2(\overline{\theta}^{\lambda}_t)\right|^2\right] \leq \cC_{D,\epsilon_h} \lambda^2/2,
\]
where 
\begin{align}\label{eq:cDexp}
\begin{split}
\cC_{D,\epsilon_h}
&\coloneqq 2^{6((l+1)(1/\epsilon_h+1)+8}\cK_{\nabla,\epsilon_h}^2\max\{\cL_0^2d/\beta,(1+a+K_h)^4+16(1+d/\beta)^2\}\\
&\qquad \times\cc_{\lceil2((l+1)/\epsilon_h+l)\rceil} (1+1/a)\left(1+\E\left[|\theta_0|^{2\lceil2((l+1)/\epsilon_h+l)\rceil}\right]\right)
\end{split}
\end{align}
with $\cK_{\nabla,\epsilon_h}\coloneqq (10\sqrt{2}+4\epsilon_h^{-1})(l+1)^2\max\{K_H,L,K_h,a,\|\nabla h_{\lambda}(0)\|,1\}$, $\cL_0\coloneqq 2a+4K_H+(l+1)(2K_h+a)$, and $\cc_p$, $p\in [2, \infty)\cap {\N}$, given in \eqref{eq:2pthmmtexpconst}.

Similarly, by using the expression of $F_3$ in \eqref{eq:decompterms}, we have, for any $0<\lambda\leq \lambda_{\max}\leq 1$, $n\in\N_0$, $t\in(n ,n+1 ]$, that
\[
\E \left[\left| \nabla h_{\lambda}(\overline{\theta}^{\lambda}_t) F_3(\overline{\theta}^{\lambda}_t)\right|^2\right] \leq \lambda^2 \left(\E \left[\| \nabla h_{\lambda}(\overline{\theta}^{\lambda}_t)\|^4\right]\right)^{1/2}\left(\E \left[|  h_{\lambda}(\overline{\theta}^{\lambda}_n)|^4\right]\right)^{1/2},
\]
which implies, by using \eqref{eq:nablahlbdaue} with $\epsilon=2$, Lemma \ref{lem:hlambdaestimates}-\ref{lem:hlambdaestimates2}, and Lemma \ref{lem:2ndpthmmt}, that
\[
\E \left[\left| \nabla h_{\lambda}(\overline{\theta}^{\lambda}_t) F_3(\overline{\theta}^{\lambda}_t)\right|^2\right] \leq \cC_{D,\epsilon_h} \lambda^2/2
\]
with $\cC_{D,\epsilon_h} >0$ given in \eqref{eq:cDexp}. Combining the two upper bounds yields the result.
\end{proof}

\begin{proof}[\textbf{Proof of Lemma \ref{lem:decomptermsrt}}]\label{lem:decomptermsrtproof}
For fixed $\theta, \overline{\theta} \in\R^d$, we follow the same argument as in the proof of Lemma~\ref{lem:hlambdaestimates}-\ref{lem:hlambdaestimates3} up to \eqref{eq:hlbdftoc} to obtain that
\begin{align*}
|h_{\lambda}(\theta)  - h_{\lambda}(\overline{\theta})-\nabla h_{\lambda}(\overline{\theta})(\theta - \overline{\theta})| 
&= \left|\int_0^1 ( \nabla h_{\lambda}(t\theta+(1-t)\overline{\theta}) - \nabla h_{\lambda}(\overline{\theta})) \,\rmd t (\theta-\overline{\theta})\right|\\
&\leq  2^{(l+1)(1/\epsilon_h+1)-2}\cL_{\nabla,\epsilon_h}(1+|\theta|+|\overline{\theta}|)^{(l+1)(1/\epsilon_h+1)-2} |\theta-\overline{\theta}|^2,
\end{align*}
where the inequality holds due to Lemma \ref{lem:hlambdaestimates}-\ref{lem:hlambdaestimates3} with $\cL_{\nabla,\epsilon_h} \coloneqq (10\sqrt{2}+4\epsilon_h^{-1})(l+1)^2\max\{K_H,L,K_h,a\}$. This, together with the expression of $r_t$ in \eqref{eq:decompterms}, yields, for any $0<\lambda\leq \lambda_{\max}\leq 1$, $n\in\N_0$, $t\in(n ,n+1 ]$, that
\begin{align}\label{eq:rtub}
\E \left[| r_t(\overline{\theta}^{\lambda}_t)  |^2\right] 
&= \E\left[\left| h_{\lambda}(\overline{\theta}^{\lambda}_n) - h_{\lambda}(\overline{\theta}^{\lambda}_t) - \nabla h_{\lambda}(\overline{\theta}^{\lambda}_t)  (\overline{\theta}^{\lambda}_n-\overline{\theta}^{\lambda}_t)\right|^2\right] \nonumber\\
&\leq  2^{2(l+1)(1/\epsilon_h+1)-4}\cL_{\nabla,\epsilon_h}^2\E \left[  \left(1+|\overline{\theta}^{\lambda}_n|+|\overline{\theta}^{\lambda}_t|  \right)^{2(l+1)(1/\epsilon_h+1)-4}|\overline{\theta}^{\lambda}_n - \overline{\theta}^{\lambda}_t|^4\right]\nonumber\\
\begin{split}
&\leq 2^{2(l+1)(1/\epsilon_h+1)-4}\cL_{\nabla,\epsilon_h}^2\left(\E \left[  \left(1+|\overline{\theta}^{\lambda}_n|+|\overline{\theta}^{\lambda}_t|  \right)^{4(l+1)(1/\epsilon_h+1)-8}\right]\right)^{1/2}  \\
&\qquad \times\left(\E \left[|\overline{\theta}^{\lambda}_n - \overline{\theta}^{\lambda}_t|^8\right]\right)^{1/2},
\end{split}
\end{align}
where the last inequality holds due to Cauchy-Schwarz inequality. Then, by using Lemma \ref{lem:2ndpthmmt}, we obtain that
\begin{align}\label{eq:rtub1}
\begin{split}
&\left(\E \left[  \left(1+|\overline{\theta}^{\lambda}_n|+|\overline{\theta}^{\lambda}_t|  \right)^{4(l+1)(1/\epsilon_h+1)-8}\right]\right)^{1/2}\\
&\leq\left(\E \left[  \left(1+|\overline{\theta}^{\lambda}_n|+|\overline{\theta}^{\lambda}_t|  \right)^{2\lceil2((l+1)/\epsilon_h+l)\rceil}\right]\right)^{1/2}\\
&\leq 3^{2(l+1)(1/\epsilon_h+1)-1/2}\left(\cc_{\lceil2((l+1)/\epsilon_h+l)\rceil}  (1+1/a) \left(1+\E\left[|\theta_0|^{2\lceil2((l+1)/\epsilon_h+l)\rceil}\right]\right)\right)^{1/2}.
\end{split}
\end{align}
Moreover, by using the definition of $(\overline{\theta}^{\lambda}_t)_{t \geq 0 }$ in \eqref{eq:ktulaproc} and by using Lemma \ref{lem:hlambdaestimates}-\ref{lem:hlambdaestimates2}, we can deduce that
\begin{equation}\label{eq:rtub2}
\E \left[|\overline{\theta}^{\lambda}_n - \overline{\theta}^{\lambda}_t|^{8}\right]\leq \cC_{\mathrm{OSE}}\lambda^{4},
\end{equation}
where $\cC_{\mathrm{OSE}} \coloneqq 2^{24}((1+a+K_h)^8+2^8(1+d/\beta)^4)\cc_{\lceil2((l+1)/\epsilon_h+l)\rceil}  (1+1/a) \left(1+\E\left[|\theta_0|^{2\lceil2((l+1)/\epsilon_h+l)\rceil}\right]\right)$. Finally, substituting \eqref{eq:rtub1} and \eqref{eq:rtub2} into \eqref{eq:rtub} yields, for any $0<\lambda\leq \lambda_{\max}\leq 1$, $n\in\N_0$, $t\in(n ,n+1 ]$, that
\[
\E \left[| r_t(\overline{\theta}^{\lambda}_t)  |^2\right] \leq \cC_{D,\epsilon_h}  \lambda^2/2,
\]
where $\cC_{D,\epsilon_h} >0$ is given in \eqref{eq:cDexp}.

Furthermore, by using the expression of $\overline{r}_t$ in \eqref{eq:decompterms}, Lemma \ref{lem:hlambdaestimates}-\ref{lem:hlambdaestimates4} and Lemma \ref{lem:2ndpthmmt}, we obtain, for any $0<\lambda\leq \lambda_{\max}\leq 1$, $n\in\N_0$, $t\in(n ,n+1 ]$, that
\[
\E \left[| \overline{r}_t(\overline{\theta}^{\lambda}_t)  |^2\right] \leq \E \left[\left| h_{\lambda}(\overline{\theta}^{\lambda}_t) - h(\overline{\theta}^{\lambda}_t)\right|^2\right] \leq \cC_{D,\epsilon_h}  \lambda^2/2
\]
with $\cC_{D,\epsilon_h} >0$ is given in \eqref{eq:cDexp}. This completes the proof.
\end{proof}

\newcommand{\pnl}{\hat{\pi}^\lambda_n}
\newcommand{\pt}{\hat{\pi}^\lambda_t}

\section{More details on functional equality}
\label{appen:changeintnder}
Recall that $\pi^{\lambda}_t$ denotes the density of $\mathcal{L}(\overline{\theta}^{\lambda}_t) $ for all $n\in\N_0$, $t\in[n,n+1]$, where $(\overline{\theta}^{\lambda}_t)_{t \geq 0 }$ is the continuous-time interpolation of kTULA~\eqref{eq:ktula}-\eqref{eq:ktulah} defined in \eqref{eq:ktulaproc}. In order to obtain the functional equality~\eqref{eq:changeofKLint}, we provide in the following lemma upper estimates for $\pi^{\lambda}_t$ together with the estimates for its log-gradient and log-Hessian. These estimates play the same role as the ones in \cite[Lemmas A5, A7, and A8]{lytras2025taming}, thus, we can proceed with the same arguments as those in \cite[Section 9]{lytras2025taming} to obtain the desired equality~\eqref{eq:changeofKLint}.

\begin{lemma}\label{lem:ueforintnderchange} Let Assumptions \ref{asm:initialc}, \ref{asm:polylip}, and \ref{asm:dissip} hold. Then, we have, for all  $0<\lambda\leq \lambda_{\max}\leq 1/(6\cL_0)^{1/(1-\epsilon_h)}$ with $\lambda_{\max}$ given in~\eqref{eq:stepsizemax}, $n\in \N_0$, $t\in (n,n+1]$, $\theta\in \R^d$, that
\begin{align}
{\pi}^{\lambda}_t(\theta) 				&\leq \overline{\cC}_0 e^{-\overline{\cC}_1 |\theta|^2}, \label{eq:ueforintnderchangeest1}\\
|\nabla \log {\pi}^{\lambda}_t(\theta)|		&\leq \overline{\cC}_0 (1+|\theta| ), \label{eq:ueforintnderchangeest2}\\
\|\nabla^2 \log {\pi}^{\lambda}_t(\theta)\|	&\leq \overline{\cC}_0 (1+|\theta|^2 ), \label{eq:ueforintnderchangeest3}
\end{align}
where $\overline{\cC}_0,\overline{\cC}_1>0$.
\end{lemma}

\begin{proof} First, we show \eqref{eq:ueforintnderchangeest1} holds by induction. We note that $\pi^\lambda_0$ decays exponentially by Assumption \ref{asm:initialc}, i.e., there exist $\overline{\cC}_2, \overline{\cC}_3>0$ such that for all $\theta\in\R^d$,
\[
\pi^\lambda_0(\theta)\leq \overline{\cC}_2 e^{-\overline{\cC}_3|\theta|^2}.
\]
For $n\in\N$, we assume $\pi^\lambda_n$ decays exponentially, i.e., there exist $\overline{\cC}_4, \overline{\cC}_5>0$ such that for all $\theta\in\R^d$,
\begin{equation}\label{eq:inductionasm}
\pi^\lambda_n(\theta)\leq \overline{\cC}_4 e^{-\overline{\cC}_5|\theta|^2}. 
\end{equation}
We proceed to show that $\pi^\lambda_{n+1}$, $n \in \N$, decays exponentially. To this end, recall the definition of $f$ given in \eqref{eq:deff}. By using Lemma~\ref{lem:hlambdaestimates}-\ref{lem:hlambdaestimates3} and by $0<\lambda\leq \lambda_{\max}\leq 1/(6\cL_0)^{1/(1-\epsilon_h)}$, we have, for any $\theta\in\R^d$, that
\[
\|\lambda  \nabla h_{\lambda}(\theta)\|\leq 1/2,
\]
which implies that $f$ is a bi-Lipschitz mapping, i.e.,
\begin{equation}\label{eq:fbilipappen}
 \cI_d/2\preceq \nabla f \preceq 3 \cI_d/2.
\end{equation}
Since $f$ defined in \eqref{eq:deff} is twice continuously differentiable, by using \cite[Theorem A]{gordon1972diffeomorphisms} and the inverse function theorem, we deduce that $f$ is bijective and there exists a twice continuously differentiable inverse map denoted by $f^{-1}$ with $f^{-1} :\R^d\rightarrow \R^d$ such that 
\[
\nabla f^{-1}=(\nabla f)^{-1}.
\]
We note that $f^{-1}$ is Lipschitz continuous due to \eqref{eq:fbilipappen}. Then, for $Z =f(\overline{\theta}^{\lambda}_n)$, we apply the transformation of variables to obtain that
\[
\psi_n^{\lambda}(z) = \frac{\pi_n^{\lambda}(f^{-1}(z))}{\rmdet(\nabla f(f^{-1}(z)))} \leq \frac{\pi_n^{\lambda}(f^{-1}(z))}{(\Lambda_{\min}(\nabla f(\theta)))^d} \leq 2^d \pi_n^{\lambda}(f^{-1}(z)) \leq  2^d\overline{\cC}_4 e^{-\overline{\cC}_5|f^{-1}(z)|^2}\leq \overline{\cC}_6 e^{-\overline{\cC}_7|z|^2},
\]
for some $\overline{\cC}_6, \overline{\cC}_7>0$, where $\Lambda_{\min}(\nabla f(\theta))$ denotes the smallest eigenvalue of $\nabla f(\theta)$, the second inequality holds due to \eqref{eq:fbilipappen}, the third inequality holds due to the induction assumption \eqref{eq:inductionasm} and the last inequality holds due to the fact that, for all $z\in\R^d$,
\[
||z|-|f(0)||\leq |f(f^{-1}(z))-f(0)|\leq (3/2)|f^{-1}(z)|.
\]
By using \eqref{eq:ktulaproc} and by noticing that $\pi_{n+1}^{\lambda} =\psi_n^{\lambda}* \phi^{\lambda}_{\beta}$ with $\psi_n^{\lambda}$ denoting the density of $f(\overline{\theta}^{\lambda}_n)$, $n \in \N_0$, and $\phi^{\lambda}_{\beta}$ denoting the density of the normal distribution with mean 0 and variance $2\lambda\beta^{-1}\cI_d$, we have, for any $\theta\in\R^d$, that
\begin{align}\label{eq:pilambdan+1ub}
\begin{split}
\pi_{n+1}^{\lambda}(\theta) =(\psi_n^{\lambda}* \phi^{\lambda}_{\beta})(\theta)
&=\int_{\R^d} \psi_n (y) \phi^{\lambda}_{\beta}(\theta-y) \rmd y \\
&\leq \int_{\R^d} \overline{\cC}_6 e^{-\overline{\cC}_7|y|^2}(4 \pi\lambda/\beta)^{-d/2} e^{-(\beta/4\lambda)|\theta-y|^{2}} \rmd y \\
&\leq \overline{\cC}_8e^{-\overline{\cC}_9|\theta|^{2}},
\end{split}
\end{align}
where $\overline{\cC}_8,\overline{\cC}_9>0$. Hence, by induction, we can conclude that $\pi^\lambda_n$, $n \in \N_0$, decays exponentially. To show this property holds for $\pi^\lambda_t$, $n\in \N_0$, $t\in(n,n+1)$, we define, for each $t$, a function $\overline{f}_t:\R^d\to\R^d$ given by, for any $\theta\in\R^d$,
\begin{equation}\label{eq:defoverlinef}
\overline{f}_t(\theta)=\theta-\lambda(t-n)h_{\lambda}(\theta).
\end{equation}
Then, we can use the same argument as that in \eqref{eq:pilambdan+1ub} to obtain the desired result, namely, $\pi^\lambda_t$ decays exponentially. This completes the proof for \eqref{eq:ueforintnderchangeest1}.

Then, we show the second inequality \eqref{eq:ueforintnderchangeest2} holds. Following the proof of Lemma \ref{lem:ktuladensitygrowth}, we obtain, for any $n \in \N_0$, $t\in(n,n+1]$, $\theta\in\R^d$, that 
\[
|\nabla \log \pi^{\lambda}_t(\theta)| \leq (\lambda(t-n))^{-1}(\widetilde{\cc}_1|\theta|+\widetilde{\cc}_2)\leq \overline{\cC}_0 (1+|\theta| ),
\]
where $\widetilde{\cc}_1 \coloneqq 2\beta  $ and $\widetilde{\cc}_2 \coloneqq 2\beta  ( \E\left[|\theta_0|^2\right] +\cc_0(1+1/a))^{1/2}$ with $\cc_0$ given in Lemma \ref{lem:2ndpthmmt}-\ref{lem:2ndpthmmti}.

To prove \eqref{eq:ueforintnderchangeest3}, we use the definition of $\overline{f}_t$ in \eqref{eq:defoverlinef} and notice that $\pi_t^{\lambda} =\overline{\psi}_t^{\lambda}* \overline{\phi}^{\lambda}_{\beta}$ with $\overline{\psi}_t^{\lambda}$ denoting the density of $\overline{f}_t(\overline{\theta}^{\lambda}_n)$, $n \in \N_0$, $t\in(n,n+1]$, and $\overline{\phi}^{\lambda}_{\beta}$ denoting the density of the normal distribution with mean 0 and variance $2\lambda (t-n)\beta^{-1}\cI_d$. Then, straight forward calculations yield, for any $\theta\in\R^d$, that
\begin{align}\label{eq:densityloghessub}
\|\nabla^2\log {\pi}^{\lambda}_t(\theta)\|
&\leq \left\|\frac{\nabla^2 {\pi}^{\lambda}_t(\theta)}{{\pi}^{\lambda}_t(\theta)}\right\|
+\left\|\frac{\nabla {\pi}^{\lambda}_t(\theta) (\nabla {\pi}^{\lambda}_t(\theta))^\cT}{({\pi}^{\lambda}_t(\theta))^2}\right\|\nonumber \\
&= \left\|\frac{\nabla^2 (\overline{\psi}_t^{\lambda}*\overline{\phi}^{\lambda}_{\beta}) (\theta)}{(\overline{\psi}_t^{\lambda}*\overline{\phi}^{\lambda}_{\beta})(\theta)}\right\|+|\nabla \log {\pi}^{\lambda}_t(\theta)|^2.
\end{align}
We note that the second therm on the RHS of \eqref{eq:densityloghessub} can be controlled by \eqref{eq:ueforintnderchangeest2}. To upper bound the first term on the RHS of \eqref{eq:densityloghessub}, we write, for any $\theta\in\R^d$, that
\begin{align}\label{eq:densityloghessub1}
\left\|\frac{\nabla^2 (\overline{\psi}_t^{\lambda}*\overline{\phi}^{\lambda}_{\beta}) (\theta)}{(\overline{\psi}_t^{\lambda}*\overline{\phi}^{\lambda}_{\beta})(\theta)}\right\|
&= \left\|\frac{\int_{\R^d} \nabla \log \overline{\psi}_t^{\lambda}(y)( \nabla \log \overline{\phi}^{\lambda}_{\beta}(\theta-y))^{\cT} \overline{\psi}_t^{\lambda}(y) \overline{\phi}^{\lambda}_{\beta}(\theta-y) \rmd y }{\int_{\R^d} \overline{\psi}_t^{\lambda}(y)\overline{\phi}^{\lambda}_{\beta}(\theta-y) \rmd y}\right\| \nonumber\\
&\leq \overline{\cC}_{10} \frac{\int_{\R^d} (1+|\theta|+|y|) \overline{\psi}_t^{\lambda}(y)\overline{\phi}^{\lambda}_{\beta}(\theta-y)\rmd y}{\int_{\R^d} \overline{\psi}_t^{\lambda}(y)\overline{\phi}^{\lambda}_{\beta}(\theta-y) \rmd y},
\end{align}
where $\overline{\cC}_{10}>0$ and the last inequality holds due to \eqref{eq:ueforintnderchangeest2}. We can then obtain an upper and a lower bound for the numerator and the denominator of the fraction on the RHS of \eqref{eq:densityloghessub1}, respectively, which can be achieved by using the same argument as that in \cite[Lemma A.7]{lytras2025taming}. Finally, we notice that, for any $\theta\in\R^d$, the following inequalities
\[
\left\|\frac{\nabla^2 (\overline{\psi}_t^{\lambda}*\overline{\phi}^{\lambda}_{\beta}) (\theta)}{(\overline{\psi}_t^{\lambda}*\overline{\phi}^{\lambda}_{\beta})(\theta)}\right\|
\leq \overline{\cC}_{11}(1+|\theta|),
\quad
|\nabla \log {\pi}^{\lambda}_t(\theta)|^2 \leq \overline{\cC}_{12}(1+|\theta|^2)
\]
hold for some $\overline{\cC}_{11},\overline{\cC}_{12}>0$, then, substituting the above upper bounds into \eqref{eq:densityloghessub} yields the desired result.
\end{proof}

\newpage
\section{Analytic Expression of Constants}
\label{appen:constexp}
{\footnotesize
\begin{tabularx}{\textwidth}[H]
{ 
	>{\hsize=0.5\hsize\linewidth=\hsize}X
	>{\hsize=0.3\hsize\linewidth=\hsize}X|
	>{\centering\hsize=1\hsize\linewidth=\hsize}X|
	>{\hsize=2.2\hsize\linewidth=\hsize}X
} 
\hline
\hline
{CONSTANT} & & DEPENDENCE ON $d,\beta$& FULL EXPRESSION\\
\hline
Theorem \ref{thm:mainthm}		& $\lambda_{\max}$		&  --- 						&$\min\{  1,1/(8a), 1/(6(2a+4K_H+(l+1)(2K_h+a)))^{1/(1-\epsilon_h)}\}$\\
Corollary \ref{crl}				& $C_0$				&  --- 						&$3C_{\rmLS}/2$\\
Corollary \ref{crl:eer}			& $C_1$  				&$\mathcal{O}(\beta (1+d/\beta)^{2(l+1)(1/\epsilon_h+1)+1})$					&$ (40\beta/(3C_{\rmLS})) (\cC_{D,\epsilon,\epsilon_h}\cC_{J,\epsilon_h}^{1-\epsilon/2}+2\cC_{D,\epsilon_h} )$ \\
						& $C_2$  				&  ---  					&$(2C_{\rmLS})^{1/2}$ \\
						& $C_3$  				&$\mathcal{O}((1+d/\beta)^{(l+1)/2})$ 	&$2^l K_hC_2\rmKL(\pi^{\lambda}_0\|\pi_{\beta})^{1/2}( (\E[|\theta_0|^{2l+2}]+ \cc_{l+1}(1+1/a))^{1/2}+(2(b+(d+2l)/\beta)/a)^{(l+1)/2}+1)$ \\
						& $C_4$  				&$\mathcal{O}(\beta^{1/2} (1+d/\beta)^{(l+1)(1/\epsilon_h+3/2)+1/2})$				&$C_2C_1^{1/2}( (\E[|\theta_0|^{2l+2}]+ \cc_{l+1}(1+1/a))^{1/2}+(2(b+(d+2l)/\beta)/a)^{(l+1)/2}+1)$ \\
						& $C_5$  				&$\mathcal{O}(d\log((1+\beta/d)(1+d/\beta)^{l/2})$								&$\frac{d}{2}\log\left(\frac{K_H e (1+4\max\{\sqrt{b/a},\sqrt{2d/(\beta K_H)}\})^l}{a}\left(\frac{\beta b}{d}+1\right)\right)+\log 2$ \\
\hline	
Lemma \ref{lem:hlambdaestimates}	& $\cL_0$ 				&  --- 						&$2a+4K_H+(l+1)(2K_h+a)$\\
						& $\cL_{\nabla,\epsilon_h}$	&  --- 						&$ (10\sqrt{2}+4\epsilon_h^{-1})(l+1)^2\max\{K_H,L,K_h,a\}$\\
\hline
Lemma \ref{lem:2ndpthmmt} 		& $\cc_0$ 				&$\mathcal{O}(1+d/\beta)$		& $2b+8K_h^2+2\beta^{-1}d$\\
						& $\cc_p$ 				&$\mathcal{O}((1+d/\beta)^p)$	& $(1+2/a)^{p-1} (1+ 2b+8K_h^2)^p(1+2^{2p-1}p(2p-1) d\beta^{-1})  +2^{2p-4}(2p(2p-1))^{p+1}(\max\{1,d\beta^{-1} \})^p+a(\max\{1,2^{2p}p(2p-1)d(a\beta)^{-1}\})^{p}.$\\
\hline
Lemma \ref{lem:ktuladensitygrowth}	& $\widetilde{\cc}_1$ 		&$\mathcal{O}(\beta)$			& $2\beta $\\
						& $\widetilde{\cc}_2$ 		&$\mathcal{O}(\beta+d)$		& $2\beta ( \E\left[|\theta_0|^2\right] +\cc_0(1+1/a))^{1/2}$\\
\hline
Lemma \ref{lem:decomptermsF1}	& $\cK_{\nabla,\epsilon_h}$ 	&--- 						& $(10\sqrt{2}+4\epsilon_h^{-1})(l+1)^2\max\{K_H,L,K_h,a,\|\nabla h_{\lambda}(0)\|,1\}$\\
						& $\cC_{D,\epsilon,\epsilon_h}$ &$\mathcal{O}(\beta^{\epsilon-2}(1+d/\beta)^{(l+1)/\epsilon_h+l+\epsilon/4})$																									& $\left(1+\E\left[|\theta_0|^{\max\{2\lceil 4((l+1)/\epsilon_h+l)/\epsilon\rceil,4\}}\right]\right)^{\epsilon/2}\cK_{\nabla,\epsilon_h}^2\beta^{\epsilon-2}(1+1/a)^{3\epsilon/4}\cc_2^{\epsilon/4}\cc_{\lceil 4((l+1)/\epsilon_h+l)/\epsilon\rceil}^{\epsilon/4}2^{2(l+1)(1/\epsilon_h+1)+13\epsilon/4}
$\\
\hline
Lemma \ref{lem:ubforfisherinfo}		& $\cC_{J,\epsilon_h}$ 		&$\mathcal{O}(d^2(1+d/\beta)^{(l+1)/\epsilon_h+l}+d\beta)$		& $J_0+4\cC_{\psi,\epsilon_h}/\cL_0^2+ 6d\beta\cL_0$\\
						& $\cC_{\psi,\epsilon_h}$ 	&$\mathcal{O}(d^2(1+d/\beta)^{(l+1)/\epsilon_h+l})$			& $\left(1+\E\left[|\theta_0|^{2\lceil (l+1)(1/\epsilon_h+1)-2\rceil}\right]\right)2^{4(l+1)(1/\epsilon_h+1)-1} d^2\cL_{\nabla,\epsilon_h}^2(1+1/a)\cc_{\lceil (l+1)(1/\epsilon_h+1)-2\rceil}$\\
\hline
Lemma \ref{lem:decomptermsF2n3}	& $\cC_{D,\epsilon_h} $ 		&$\mathcal{O}( (1+d/\beta)^{2(l+1)(1/\epsilon_h+1)+1})$			& $\left(1+\E\left[|\theta_0|^{2\lceil2((l+1)/\epsilon_h+l)\rceil}\right]\right)2^{6((l+1)/\epsilon_h+l)+8}\cK_{\nabla,\epsilon_h}^2(1+1/a)\cc_{\lceil2((l+1)/\epsilon_h+l)\rceil} \max\{\cL_0^2d/\beta,(1+a+K_h)^4+16(1+d/\beta)^2\} $\\
\hline
\hline
\end{tabularx}
}

\bibliographystyle{plainnat}

\bibliography{references}
\end{document}